\documentclass[reqno,11pt, a4paper]{amsart}

\usepackage[T1]{fontenc}

\usepackage{amssymb}
\usepackage{amsmath}
\usepackage{amsfonts}
\usepackage{amsthm}
\usepackage{mathabx}

\usepackage{mathtools}

\usepackage{fullpage}
\usepackage[dvipsnames]{xcolor}
\usepackage{graphicx}
\usepackage{float}
\usepackage{soul}

\usepackage{mdframed}
\usepackage{parskip}
\usepackage{enumitem}

\usepackage{booktabs,array}
\usepackage{url} 
\usepackage[colorlinks=true,
linkcolor=black]{hyperref}

\usepackage{pgfplots}
\pgfplotsset{compat=1.6}

 \usepackage[sort&compress,square,comma,numbers]{natbib}

\usetikzlibrary{calc,quotes,angles,arrows.meta,positioning, fillbetween}
\usetikzlibrary{decorations.pathreplacing,decorations.markings,snakes,calligraphy}
\usetikzlibrary{intersections, pgfplots.fillbetween}

\usepackage{float}

\usepackage{caption}
\usepackage{subcaption}

\theoremstyle{plain}%
\newtheorem{theorem}{Theorem}[section]
\newtheorem{lemma}[theorem]{Lemma}
\newtheorem{proposition}[theorem]{Proposition}

\newtheorem*{conjecture*}{Conjecture}

 \numberwithin{equation}{section}

\theoremstyle{definition}

\theoremstyle{remark}
\newtheorem{remark}[theorem]{Remark}

 \let \leq \leqslant
 \let \geq \geqslant

\DeclareMathOperator{\area}{area}
\DeclareMathOperator{\tr}{tr}

\DeclareMathOperator{\support}{supp}

\DeclareMathOperator{\dist}{dist}

\definecolor{detailcolor00}{rgb}{0.4405, 0.204, 0.343}
\definecolor{detailcolor01}{rgb}{0.546, 0.215, 0.352}
\definecolor{detailcolor02}{rgb}{0.675, 0.247, 0.387} 
\definecolor{detailcolor03}{rgb}{0.775, 0.317, 0.455}
\definecolor{detailcolor04}{rgb}{0.830, 0.421, 0.553} 
\definecolor{detailcolor05}{rgb}{0.831, 0.533, 0.663}
\definecolor{detailcolor06}{rgb}{0.779, 0.619, 0.775}
\definecolor{detailcolor07}{rgb}{0.724, 0.694, 0.827}
\definecolor{detailcolor08}{rgb}{0.687, 0.770, 0.880}
\definecolor{detailcolor09}{rgb}{0.671, 0.839, 0.904}
\definecolor{detailcolor10}{rgb}{0.659, 0.872, 0.882}

\definecolor{detailcolor00}{RGB}{37,37,60}
\definecolor{detailcolor05}{RGB}{18,52,86}
\definecolor{detailcolor02}{RGB}{105,41,89}
\definecolor{detailcolor03}{RGB}{165,141,0} 
\definecolor{detailcolor01}{RGB}{255,102,51} 
\definecolor{detailcolor06}{RGB}{255,102,51} 
\definecolor{detailcolor07}{RGB}{232,224,224} 
\definecolor{detailcolor08}{RGB}{232,224,224}
\definecolor{lightgraydove}{RGB}{18,52,86}

\usepackage{scalerel,stackengine}

\ifdim\overfullrule>0pt 
  \usepackage{environ}

  \NewEnviron{tikzpicture}{%
    \begin{pgfpicture}
    \pgfpathrectanglecorners{\pgfpointorigin}{\pgfpoint{3cm}{3cm}}%
     \pgfusepath{stroke}\end{pgfpicture}%
  }
\fi

\makeatletter
\newcommand{\hathat}[1]{%
\begingroup%
  \let\macc@kerna\z@%
  \let\macc@kernb\z@%
  \let\macc@nucleus\@empty%
  \hat{\raisebox{.3ex}{\vphantom{\ensuremath{#1}}}\smash{\hat{#1}}}%
\endgroup%
}

\newcommand{\smallhathat}[1]{%
\begingroup%
  \let\macc@kerna\z@%
  \let\macc@kernb\z@%
  \let\macc@nucleus\@empty%
  \hat{\raisebox{.05ex}{\vphantom{\ensuremath{#1}}}\smash{\hat{#1}}}%
\endgroup%
}

\newcommand{\smallsmallhathat}[1]{%
\begingroup%
  \let\macc@kerna\z@%
  \let\macc@kernb\z@%
  \let\macc@nucleus\@empty%
  \hat{\raisebox{-.2ex}{\vphantom{\ensuremath{#1}}}\smash{\hat{#1}}}%
\endgroup%
}
\makeatother

\makeatletter
\newcommand{\subalign}[1]{%
  \vcenter{%
    \Let@ \restore@math@cr \default@tag
    \baselineskip\fontdimen10 \scriptfont\tw@
    \advance\baselineskip\fontdimen12 \scriptfont\tw@
    \lineskip\thr@@\fontdimen8 \scriptfont\thr@@
    \lineskiplimit\lineskip
    \ialign{\hfil$\m@th\scriptstyle##$&$\m@th\scriptstyle{}##$\hfil\crcr
      #1\crcr
    }%
  }%
}
\makeatother

\title{Phase transitions in the charged compact abelian lattice Higgs model}

\author{Malin P. Forsstr\"om}
\address[Malin P. Forsstr\"om]{Department of Mathematics, Chalmers University of Technology and University of Gothenburg, Sweden}
\email{palo@chalmers.se}

\begin{document}

\maketitle 

\begin{abstract}
	We consider the (compact) abelian lattice Higgs model with charge \( k \geq 1 \) and show, using charged Wilson~loop observables and charged versions of the Marcu--Fredenhagen ratio, that this model exhibits several distinct phase transitions. In particular, we show that if \(k = 2\), then the Marcu--Fredenhagen ratio and Wilson~loop observables together can distinguish among three distinct phases of the parameter space, and hence both can be used as order parameters for the model. 
\end{abstract}

\section{Introduction}

 
Lattice gauge theories are spin models on~\( \mathbb{Z}^d \) that arise as a natural discretization of the Yang-Mills model, which is a classical model for the standard model in physics~\citep{w1974,w1971,c2018}. From a mathematical perspective, lattice gauge theories are interesting for several reasons. First, the mathematical framework for studying the Yang-Mills model itself does not yet exist; while lattice models seem to exhibit many of the phenomena expected of their continuous counterparts~\citep{c2018}. At the same time, lattice gauge theories naturally fall into the family of classical spin models for which great progress has been made in the mathematical literature, such as percolation theory, the Ising model, and the XY model; while at the same time being fundamentally different since the random paths that often arise in spin models are here naturally replaced by random surfaces. 
Finally, we mention that lattice gauge theories also arise naturally in quantum information theory, see, e.g., \cite{ssn2021,spkl2026}.
What behavior a lattice gauge theory exhibit depends of several different properties of the model different properties of the model: the dimension and boundary conditions of the lattice considered, the group in which the spins take their values, and whether or not the model includes an external field. The recent mathematical literature on lattice gauge theories reflects this diversity. In particular, the papers~\citep{a2021, fv2025, fv2025b, f2024a, f2024b, f2025, f2022b,flv, flv2022, c2020, ca2025, ds2024} study the behavior for finite groups with no external field, the papers~\citep{gs2023, g1980, gm1982} consider \( U(1) \), and~\citep{bn1987,c2025, b1990, fs1982, k1985, k1987, os1978, fs1979, bbij1984} consider the same models with an external field. In this paper, we complement this by studying \( U(1) \) lattice gauge theory on \( \mathbb{Z}^m, \) \( m \geq 4, \) in an external field with a charge \(k, \) using both the charged Marcu--Fredenhagen ratio and charged Wilson~loops to verify the structure of its phase diagram as described in the physics literature (see, e.g., \citep{g2011}). This model has to our knowledge not been studied in the mathematical literature before except in the case \( k =1, \) when it reduces to the compact abelian lattice Higgs model, i.e., to \( U(1) \) lattice gauge theory with an external field.


For an abelian group \( G \), known as the structure or gauge group, and a unitary representation~\( \rho \) of \( G, \) the Hamiltonian corresponding to a lattice gauge theory on \( B_N \coloneqq \mathbb{Z}^m \cap [-N,N]^m \) with free boundary conditions is given by
\begin{equation}\label{eq: pure hamiltonian}
	H_{N,\beta}(\sigma) \coloneqq -\beta \sum_{p\in C_2(B_N)} \tr \rho(d\sigma(p)) ,\quad \sigma \in \Omega_1(B_N,G).
\end{equation}
Here \( \beta \geq 0, \) \( d \) is the discrete derivative \( d \sigma(p ) = \sum_{ e \in \partial p} \sigma(e) \), \( \Omega_1(B_N,G) \) is the set of all \( G \)-valued 1-forms on \( B_N ,\) or equivalently, the set of all \( G \)-values functions~\( \sigma \) defined on the set of oriented edges \( C_1(B_N) \) in \( B_N \) with the property that \( \sigma(e)  = - \sigma(-e)  \) for all \( e \in C_1(B_N), \) and \( C_2(B_N) \) is the set of oriented plaquettes in \( B_N. \) Together with a uniform reference measure \( \mu \), this describes a probability measure \( \mu_{N,\beta} \) on \( \Omega_1(B_N,G), \) and the model described by this measure is what is referred to as a~\emph{lattice gauge theory.} The corresponding expectation is denoted by \( \mathbb{E}_{N,\beta} \) and the infinite volume limit \( N \to \infty \) of the expectation by \( \langle \cdot \rangle_{\beta}. \)

The most important observables in pure gauge theories are Wilson loops, which we now introduce. Throughout this paper, we will refer to a set of oriented edges as a \emph{path}. Using the language of discrete exterior calculus, a path is equivalently a \( \{ -1,0,1 \} \)-valued \( 1 \)-chain. The \emph{Wilson~line observable} corresponding to a path \( \gamma \) is defined by 
\begin{equation*}
	W_\gamma(\sigma) \coloneqq \tr \rho(\sigma(\gamma)) \coloneqq  \tr \prod_{e \in \gamma} \rho(\sigma(e)), \quad \sigma \in \Omega_1(B_N,G).
\end{equation*}
When the edges of a path forms a loop, we refer to the 1-chain \( \gamma \) as a \emph{loop}, and in this special case the Wilson~line observable is referred to as a \emph{Wilson~loop observable}. A path that does not form a loop is called an \emph{open path}.

Throughout the paper, we will assume that \( G=U(1) \) and that  \( \rho \) is the one-dimensional representation given by \( \rho \colon \theta \mapsto e^{i \theta}\), \( \theta \in U(1). \) This is the simplest gauge group of physical relevance, as with this choice the model described by~\eqref{eq: pure hamiltonian} corresponds to electromagnetism. On \( \mathbb{Z}^2 \) and \( \mathbb{Z}^3, \) when \( G= U(1), \) Wilson~loops are known to always follow an area law~\citep[Corollary 2]{gm1982}. In contrast, on \( \mathbb{Z}^4, \) Wilson~loops are known to undergo a phase transition, known as a deconfinement transition, with Wilson~loops having area law for small \( \beta \) and perimeter law for large \( \beta\) \citep{os1978, g1980, fs1982}. Here perimeter law means that there are constants \( C,c>0 \) such that \( \langle W_\gamma \rangle_{\beta} \sim Ce^{-c|\gamma|} \) and area law means that there are constants \( C,c>0 \) such that \( \langle W_\gamma \rangle_{\beta} \sim Ce^{-c \area(\gamma)}. \)

The lattice Higgs model is the model obtained by coupling a lattice gauge theory to an external field, modeling a Higgs field. For \( G = U(1), \) this model is known as the \emph{(compact) abelian lattice Higgs model}, and has Hamiltonian given by 
\begin{equation}\label{eq: hamiltonian compact abelian lattice higgs}
	H_{N,\beta,\kappa}(\sigma) \coloneqq -\beta \sum_{p\in C_2(B_N)} \rho(d\sigma(p)) - \kappa \sum_{e \in C_1(B_N)} \rho(\sigma(e)),\quad \sigma \in \Omega_1(B_N,U(1)).
\end{equation}
We let \( \mathbb{E}_{N,\beta,\kappa} \) denote the corresponding expectation for a uniform reference measure \( \mu \), and let~\( \langle \cdot \rangle_{\beta,\kappa} \) denote the corresponding infinite volume limit \( N \to \infty\), which exists due to the Ginibre inequalities, see, e.g,~\citep{f2022b}[Section 2.6].

In the model described by~\eqref{eq: hamiltonian compact abelian lattice higgs}, using the Griffiths-Hurst-Sherman inequality twice, one easily shows that for any Wilson~loop observable \( W_\gamma, \) we have
\begin{equation*}
	\langle W_\gamma \rangle_{\beta,\kappa} \geq \prod_{e \in \gamma} \langle W_e \rangle_{\beta,\kappa} \geq \prod_{e \in \gamma} \langle W_e \rangle_{0,\kappa} = \Bigl( \frac{e^{2 \kappa}-e^{2 \kappa}}{e^{2 \kappa}+e^{-2 \kappa}}\Bigr)^{|\gamma|} = (\tanh 2\kappa)^{|\gamma|},
\end{equation*} 
and hence for any \( \kappa > 0, \) Wilson~loop observables obey a perimeter law. In particular, this implies that Wilson~loop observables cannot be used as order parameters in the lattice Higgs model in the same way as they can in a pure lattice gauge theory. As a consequence, other order perimeters have been suggested in the physics literature. In this paper, we consider the one such order parameter, originally proposed by Fredenhagen and Marcu in~\citep{fm1983}, and use current expansions and ideas from disagreement percolation to give a rigorous proof that this order parameter has a phase transition.  
To define this order parameter, let \( \gamma_n \) and \( \gamma_n' \) be two paths as in Figure~\ref{figure: Wilson lines MF}. 
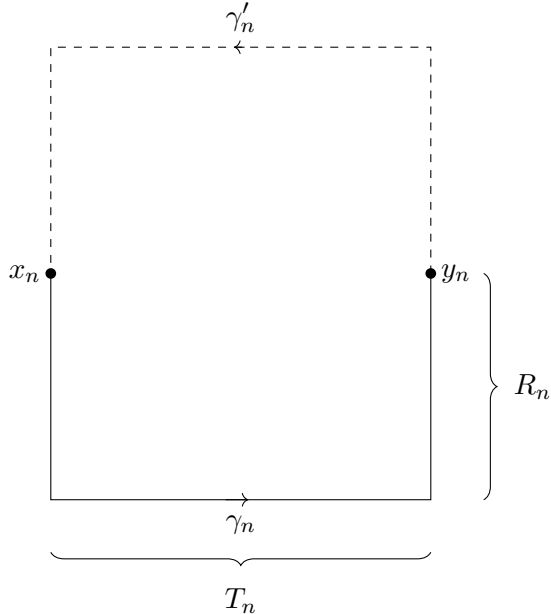
\begin{figure}[h]
	\begin{tikzpicture} 
		
		\draw (0,0) node[left] {$x_n$}  -- (0,-3)-- (5,-3) node[midway,below, yshift=-2pt] {$\gamma_n$}  -- (5,0)  node[right] {$y_n$};
		
		\draw[dashed] (5,0) -- (5,3) -- (0,3) node[midway,above,yshift=2pt] {$\gamma_n'$} -- (0,0);  
		
		\fill (0,0) circle (2pt);
		\fill (5,0) circle (2pt);
		
		\draw [decorate,decoration={brace,amplitude=5pt,mirror,raise=4.2ex}]
  (0,-3) -- (5,-3) node[midway,yshift=-3.5em]{$T_n$};
		
		\draw [decorate,decoration={brace,amplitude=5pt,mirror,raise=4.2ex}]
  (5,-3) -- (5,0) node[midway,xshift=3.5em]{$R_n$};
  
  		\draw[->] (2.3,-3) -- (2.6,-3);
  		\draw[->] (2.55,3) -- (2.41,3);
	\end{tikzpicture}
	
	\caption{The paths \( \gamma_n \) (solid) and \( \gamma_n' \) (dashed) used in the definition of the Marcu--Fredenhagen ratio.}
	\label{figure: Wilson lines MF}
	\end{figure}
	The ratio
	\begin{equation}\label{eq: ratio MF} 
		r_n \coloneqq r_n(\beta,\kappa) \coloneqq \frac{\langle W_{\gamma_n} \rangle_{\beta,\kappa}\langle W_{\gamma_n'} \rangle_{\beta,\kappa}}{\langle W_{\gamma_n+\gamma_n'} \rangle_{\beta,\kappa}},
	\end{equation}
	is then referred to as the Marcu--Fredenhagen ratio, and was introduced in~\citep{fm1983} (see also~\cite{fm1988}) to \emph{``study a sequence of states which describe a pair of charges separated by an increasing distance and which are regularized such that their energy is uniformly bounded''}. When \( R_n = R n \) and \( T_n = T n \) for some \( R,T \in \mathbb{N}, \) this ratio believed to have a phase transition between a region, referred to as the free phase, where \( \lim_{n \to \infty} r_n = 0,\) and a region referred to as the Higgs/confinement phase, where \( \liminf_{n \to \infty} r_n >0. \)	
	For an overview of rigorous results about this ratio and about the abelian lattice Higgs model in general, we refer the reader to~\citep{bn1987} (see also~\citep{fs1979} and~\citep{g2011}). 
	
The first main result of this paper, Theorem~\ref{theorem: U1 Wilson line ratios MF} below, confirms that the Marcu--Fredenhagen ratio indeed has a phase transition.
\begin{theorem}\label{theorem: U1 Wilson line ratios MF}
	Let \( m \geq 4, \) let \( \gamma_n \) and \( \gamma_n' \) be as in Figure~\ref{figure: Wilson lines MF}, let \( R_n = R n \) and \( T_n = T n \), and consider the Marcu--Fredenhagen ratio \( r_n \) as defined in~\eqref{eq: ratio MF}. Then the following holds.
	\begin{enumerate}[label=(\alph*)]
		\item\label{item: U1 confinement MF} There is \( \beta_{conf} > 0 \)  such that if \( \kappa >0  \) and \( \beta \in (0, \beta_{conf}),\) then \( \liminf_{n \to \infty} r_n >0. \)
		\item\label{item: U1 Higgs MF} There is \( \kappa_{Higgs}> 0 \) such that if \( \kappa > \kappa_{Higgs} \) and \( \beta>0, \) then \( \liminf_{n \to \infty} r_n >0. \)
		\item\label{item: U1 free MF} There is \( \beta_{free}> 0 \) and \( \kappa_{free}> 0 \) such that if \( \kappa \in (0, \kappa_{free}) \) and \( \beta>\beta_{free}, \) then \( \liminf_{n \to \infty} r_n = 0. \)
	\end{enumerate} 
\end{theorem}

\begin{remark}
 	A proof sketch of Theorem~\ref{theorem: U1 Wilson line ratios MF}\ref{item: U1 confinement MF}\ref{item: U1 Higgs MF} appear already in in~\citep{fs1979}. %
	A proof of Theorem~\ref{theorem: U1 Wilson line ratios MF}\ref{item: U1 free MF} appear also in\citep{fm1983} as well as in~\citep[Section 1.2]{b1990} and~\citep{b1990}, and a proof of Theorem~\ref{theorem: U1 Wilson line ratios MF} also appear in~\citep{k1985}. This statement does not hold if \( m = 3, \) and the arguments which work for \( m = 4 \) fail since they all use that there are \( \beta>0 \) such that the model has perimeter law, and such \( \beta \) does not exist when \( m =3. \) 
\end{remark}


A richer version of the abelian lattice Higgs model is the \emph{compact abelian lattice Higgs model with charge \( k \) matter} \citep{g2011}, described by the Hamiltonian 
\begin{equation}\label{eq: charged hamiltonian}
	H_{N, \beta, \kappa, k}(\sigma) \coloneqq -\beta  \sum_{p \in C_2(B_N)}  \rho(d\sigma(p)) - \kappa \sum_{e \in C_1(B_N)} \rho(\sigma(e))^k, \quad \sigma \in \Omega_1(B_N,\mathbb{Z}_2).
\end{equation}
Indeed, letting \( k = 0 \) (or \( \kappa = 0 \)) we recover the Hamiltonian of a pure \( U(1)\) lattice gauge theory, and letting \( k =1 \) we recover the Hamiltonian of the abelian lattice Higgs model. 
For a uniform reference measure \( \mu ,\) we let \( \mu_{N,\beta,\kappa,k} \) be the measure corresponding to~\eqref{eq: charged hamiltonian}, let \( \mathbb{E}_{N,\beta,\kappa,k} \) be the corresponding expectation, and let \( \langle \cdot \rangle_{\beta,\kappa,k} \) be its infinite volume limit.
Just as in \( U(1) \) lattice gauge theory or the abelian lattice Higgs model, it is natural to consider Wilson~loop and~line observables. Such observables can be made more general by introducing \emph{charges}, with charge \( j\geq 1 \) corresponding to the observable
	\begin{equation}
		W_{j\gamma}(\sigma) = \rho(\sigma(j\gamma)) = \rho(\sigma(\gamma))^j.
	\end{equation}
Another natural observable for this model is the \(j\)-charge analog of the Marcu--Fredenhagen ratio, which is given by
\begin{equation*}
	r_{n,k,j}(\beta,\kappa) \coloneqq \frac{\langle W_{j \gamma_n }\rangle_{\beta,\kappa,k} \langle W_{j \gamma_n' }\rangle_{\beta,\kappa,k} }{\langle W_{j (\gamma_n +\gamma_n')}\rangle_{\beta,\kappa,k} }.
\end{equation*}

The phase diagram of the abelian lattice Higgs model with charge \( \geq 2 \) is expected to be richer than that of the abelian lattice Higgs model with charge~1. For example, when~\( k = 2, \) the physics literature (see, e.g., \citep{fs1979,g2011}) suggests the existence of three distinct phases: 
\begin{enumerate}
	\item A confinement phase (\( \beta \) large and \( \kappa \) large), where \( {\langle W_{\gamma} \rangle_{\beta,\kappa,2} \sim Ce^{-c|\gamma|} }\),  \( {\langle W_{2\gamma} \rangle_{\beta,\kappa,2} \sim Ce^{-c|\gamma|},} \) and \( {\liminf_{n \to \infty} r_{n,2,2} > 0 }\).
	\item A Higgs phase (\( \beta \) small), where \( \langle W_{\gamma} \rangle_{\beta,\kappa,2} \sim Ce^{-c \area(\gamma)} \),  \(\langle W_{2\gamma} \rangle_{\beta,\kappa,2} \sim Ce^{-c|\gamma|}, \) and \( \liminf_{n \to \infty} r_{n,2,2} > 0 \).
	\item A free phase (\( \beta \) large and \( \kappa \) small), where \( \langle W_{\gamma} \rangle_{\beta,\kappa,2} \sim Ce^{-c|\gamma|} \),  \(\langle W_{2\gamma} \rangle_{\beta,\kappa,2} \sim Ce^{-c|\gamma|} \), and \( \lim_{n \to \infty} r_{n,2,2} = 0 \). 
\end{enumerate}

Our last two main results, Theorem~\ref{theorem: main result charge nondiv} and Theorem~\ref{theorem: main result charge div} below, confirm this picture for general~\( k\), and correspond to the two cases \( k \not \mid j \) and \( k \mid j \) respectively. The first of these theorems, Theorem~\ref{theorem: main result charge nondiv}, concerns the case \( k \not \mid j, \) and show that in this case, the charged Marcu--Fredenhagen ratio cannot detect phase transition. The reason for this is that in this case, Wilson~line observables have expectation zero exactly as when \( \kappa =0.\) As an important special case, it follows that when  \( k \geq 2 \) and \( j = 1, \) the regular Marcu--Fredenhagen ratio will not have a phase transition.

\begin{theorem}\label{theorem: main result charge nondiv}
	Let \(m \geq 4, \) let \( j ,k\geq 1\) be such that \( k \not \mid j,\) and let \( \beta,\kappa >0. \) Further, let  \( R,T \geq 1, \) and for \( n \geq 1, \) let  \( \gamma_n \) be a rectangular loop with side lengths \( Rn \) and \( Tn \), and let \( \beta_c \) be so that pure lattice gauge theory with \( \beta> \beta_c \) has perimeter law. 
	Then the following holds. 
	\begin{enumerate}[label=(\alph*)]
		\item For any open path \( \gamma, \) we have  \( \langle W_{j \gamma} \rangle_{\beta,\kappa,k} = 0 \), and hence
		\( r_{n,k,j} = 0 \) for all \( n \geq 1. \) \label{theorem: always zero if not divisible}
		\item There is \( \beta_{conf} > 0 \), \( C = C(\beta,\kappa)>0, \) and \( c=c(\beta,\kappa) > 0 \) such that if \( \beta \in (0,\beta_{conf}) \), then   \( \langle W_{j\gamma_n} \rangle_{\beta,\kappa,k} \leq Ce^{-c \area(\gamma_n)}\) for all \( n \geq 1. \) \label{theorem: area law}
	\item There is \( C = C(\beta,\kappa)>0, \) and \( c=c(\beta,\kappa) > 0 \), such that if \( \beta  > \beta_c \), then \( {\langle W_{j \gamma_n} \rangle_{\beta,\kappa,k} \geq Ce^{-c|\gamma_n|}}\) for all \( n \geq 1. \)  \label{theorem: always a region with perimeter law nondiv}
	\end{enumerate}
\end{theorem}

The following theorem, Theorem~\ref{theorem: main result charge div}, describes the behavior of charged, Wilson~line observables and the charged Marcu-Fredenhaen ratio in the case \( k \mid j, \) which is the case that is the most relevant for the study of phase transitions of the charged model. In particular, it shows that charged observables can be used to detect phase transitions.

\begin{theorem}\label{theorem: main result charge div}
	Let \(m \geq 4, \) let \( j ,k \geq 0\) be such that \( k \mid j, \) and let \( \beta,\kappa>0.\) Further, let  \( R,T \geq 1, \) and for \( n \geq 1, \) let  \( \gamma_n \) be a rectangular loop with side lengths \( Rn \) and \( Tn \), and let \( \beta_c \) be such that pure lattice gauge theory with \( \beta> \beta_c \) has perimeter law.  Then the following holds. 
	\begin{enumerate}[label=(\alph*)]
		\item There is \( c,C > 0 \) such that  for any path \( \gamma, \) we have \(
		\langle W_{j\gamma} \rangle_{\beta,\kappa,k } > C e^{-c|\gamma|}. \) In particular, if \( \gamma \) is a loop, then the Wilson~loop observable \( W_{j \gamma} \) has a perimeter law for all \( \beta,\kappa >0. \) \label{theorem: always perimeter law}   
		
	\item There is \( \beta_{conf}>0 \) such that if \( \beta \in (0,\beta_{conf}) \), then \( \liminf_{n \to \infty} r_{n,k,j} > 0. \)  \label{theorem: main result div confinement}

	\item For every \( \beta> 0 \) there is \( \kappa_{Higgs} = \kappa_{Higgs}(\beta)>0 \) such that if \( \kappa >\kappa_{Higgs}(\beta) \), then \( \liminf_{n \to \infty} r_{n,k,j} > 0. \)  \label{theorem: main result div higgs}

	\item \label{theorem: main result div free}
	For every \( \beta>\beta_c \) there is \( \kappa_{free} (\beta)>0 \) such that if \( \kappa<\kappa_0(\beta), \) then \( \lim_{n \to \infty} r_{n,k} = 0. \)
	
	\end{enumerate}
\end{theorem}

We note that Theorem~\ref{theorem: main result charge div} in fact implies Theorem~\ref{theorem: U1 Wilson line ratios MF}, as Theorem~\ref{theorem: U1 Wilson line ratios MF} corresponds to the special case \( j=k=1 \) of Theorem~\ref{theorem: main result charge div}.

\begin{remark}
	In the literature, there are several closely related models which are all referred to as the (compact) abelian lattice Higgs model with charge \( k \). The most general of these models is given by the Hamiltonian
	\begin{align*}
		&-\beta \sum_{p\in C_2(B_N)} \rho(d\sigma(p)) - \kappa \sum_{e = (y-x) \in C_1(B_N)} \eta_x \rho(\sigma(e))^k  \eta_y^* +\sum \eta_x\eta_x^* + \lambda \sum (\eta_x\eta_x^*-1)^2,
	\end{align*}
	where \( \beta,\kappa,\lambda>0 \), \( \sigma \in \Omega_1(B_N,U(1))\), and \(\eta \colon C_0(B_N) \to \mathbb{C}\) is such that \( \eta_x = \eta_{-x}^*. \)
	Equivalently, one can let \( \sigma \in \Omega_1(B_N,U(1))\), \(\phi \in  \Omega_0(B_N,U(1))\), and \( r \colon C_0(B_N) \to \mathbb{R}_+ \) be such that \( r_x = r_{-x} \) and consider the Hamiltonian
	\begin{align*}
		&-\beta \sum_{p\in C_2(B_N)} \rho(d\sigma(p)) - \kappa \sum_{e = (y-x) \in C_1(B_N)} r_xr_y \rho(k\sigma(e)-d\phi(e)) + \sum r_x^2 + \lambda \sum (r_x^2-1)^2.
	\end{align*}
	 Letting \( \lambda = \infty, \) one obtains the model described by the Hamiltonian
	\begin{align*}
		-\beta \sum_{p\in C_2(B_N)} \rho(d\sigma(p)) - \kappa \sum_{e \in C_1(B_N)} \rho(k \sigma(e)-d\phi(e)).
	\end{align*}
	This model is sometimes referred to as the London limit of the Abelian lattice Higgs model with charge~\( k \). Using the change of variables \( \sigma \mapsto \sigma+d\phi \), often  referred to as choosing unitary gauge, one obtains the equivalent model
	\begin{align*}
	H_{N,\beta,\kappa,\lambda}(\sigma) &\coloneqq -\beta \sum_{p\in C_2(B_N)} \rho(d\sigma(p)) - \kappa \sum_{e  \in C_1(B_N)} \rho(\sigma(e))^k 
	\end{align*}
	used in this paper.
\end{remark}

\begin{figure}[ht]
\centering
\begin{tikzpicture}[scale=1.3]
	\begin{axis}[xmin=-0.05,xmax=0.8,ymin=-0.05,ymax=0.8,xticklabel=\empty,yticklabel=\empty,axis lines = middle,xlabel=$\beta$,ylabel=$\kappa$,label style =
               {at={(ticklabel cs:1.0)}}, color=detailcolor07, ticks=none]
               
               
               
               

               \addplot[name path=A, domain=0.44:0.9, samples=100, draw=blue, draw opacity=.4]{0.2*(x-0.44)^(0.5)};

               \addplot[name path=B, domain=0.44:0.9, samples=100, draw=none]{0};

    		 \addplot [blue, opacity=.1] fill between [of=A and B];

               \addplot[name path=C, domain=0.24:0.9, samples=100, draw=blue, draw opacity=.4]{0.77 * (x-0.24)^0.17};

               \addplot[name path=D, domain=0.24:0.9, samples=100, draw=none]{.77};

    		 \addplot [blue, opacity=.1] fill between [of=C and D];

  		\end{axis}
  		

  		\draw (5.4,0.50) node[] {\tiny \(\lim_{n \to \infty}r_{n,2,2} = 0\)};
  		
  		\fill[blue, opacity=.1] (0.41,0.34) -- (2.34,0.34) -- (2.34,5.5) -- (0.41,5.5) -- cycle;
  		
  		\draw (1.7,2.4) node[] {\tiny \( \langle W_{\gamma} \rangle_{\beta,\kappa,2} \lesssim Ce^{-c \area(\gamma)}\)}; 
  		
  		\draw (3.3,4.7) node[] {\tiny \( \liminf_{n \to \infty}r_{n,2,2} >0 \)};
  		
  		\fill[orange, opacity=.1] (0.41,0.34) -- (2.8,0.34) -- (2.8,5.5) -- (0.41,5.5) -- cycle;
  		\draw[orange, opacity=.4] (2.8,0.34) -- (2.8,5.5);

  		\fill[red, opacity=.1] (4,0.34) -- (6.86,0.34) -- (6.86,5.5) -- (4,5.5) -- cycle;
  		\draw[red, opacity=.4] (4,0.34) --  (4,5.5);
  		
  		\draw (5.4,2.4) node[] {\tiny \( \langle W_{\gamma} \rangle_{\beta,\kappa,2} \sim Ce^{-c |\gamma|}\)};
  		 
\end{tikzpicture} 
	\caption{A summary of the reuslts of Theorem~\ref{theorem: main result charge nondiv} and Theorem~\ref{theorem: main result charge div} in the special case~\( k = 2 .\)}\label{}
\end{figure}
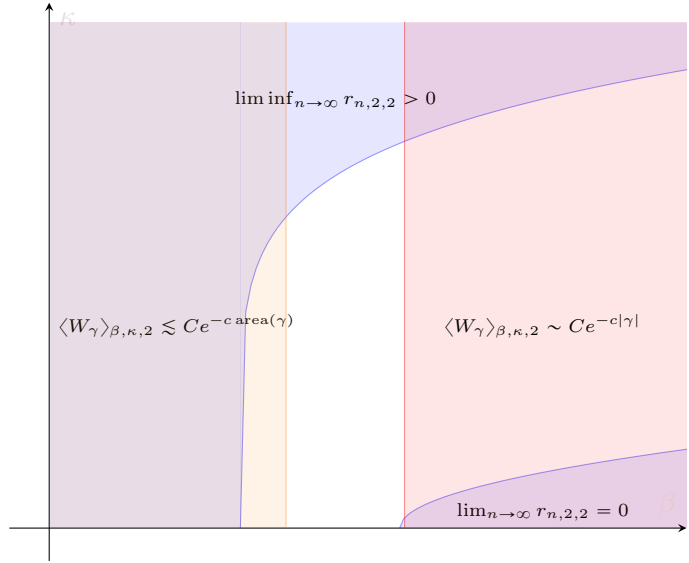
 
\subsection{Related works}

In~\citep{f2024b}, we showed that the Marcu--Fredenhagen ratio has a phase transition in the Ising lattice Higgs model, which is the model with the same Hamiltonian as in~\eqref{eq: hamiltonian compact abelian lattice higgs} but with \( \sigma \in \Omega_1(B_N,\mathbb{Z}_2) \) instead of \( \sigma \in \Omega_1(B_N,U(1)). \) In that paper, the proofs used cluster expansions of \( \log \mathbb{E}_{N,\beta}[\gamma] \) and high temperature expansions to obtain the desired results. However, for the model considered in this paper, one cannot directly apply such cluster expansions, as this requires the structure group to be discrete. In addition, the abelian lattice Higgs model does not have a high temperature expansion analog to that corresponding to the case \( G=\mathbb{Z}_2. \)

One way to think about the abelian lattice Higgs model with charge \( k \) is to think of the Higgs field component as making \( \rho(\sigma) \) prefer to be close to the set \( \{ e^{ j 2\pi i /k}\}_{j = 0,1,\dots, k-1} \) instead of preferring to be close to \( 1 = e^0\) as in the \( k = 1 \) case. The same effect can be obtained analogously for a finite gauge group. However, in this case, not all choices of \( k \) yield distinct models, as, e.g., \( G = \mathbb{Z}_n \) and \( k = m+jn \) yields the same model for all \( j \geq 0. \)
In~\citep{a2021}, a related model was considered, which was described by the Hamiltonian
\begin{equation}\label{eq: arkas model}
	-\beta \sum_{p \in C_2(B_N)} \tr \rho(d\sigma(p)) - \kappa \sum_{e \in C_1(B_N)} \tr \rho(\sigma(e)) \rho(d\phi(e))^{-1}, \quad \sigma \in \Omega_1(B_N,G) ,\, \phi \in \Omega_0(B_N,H).
\end{equation}
Here \( G \) and \( H \) were finite groups that did not have to be equal, as long as their representations had the same dimension. In the discrete case, letting \( G = \mathbb{Z}_{kn} \) and \( H=\mathbb{Z}_k \) results in a model similar to the model considered in this paper. However, they are not equivalent, as letting \( \kappa \to \infty \) in the abelian lattice Higgs model with charge \( 2 \) results in the \( \mathbb{Z}_2 \) lattice Higgs model, while letting \( \kappa \to \infty \) in the model described by~\eqref{eq: arkas model}, for \( G = \mathbb{Z}_{2n} \) and \( H = n \) results in the pure Ising model.
The main result of~\citep{a2021} for this model was asymptotics for Wilson~loop observables in the dilute gas limit, meaning that both \( \beta \) and either \( \kappa \) or \( \kappa^{-1} \) had go to infinity as functions of \( |\gamma| \). In particular, the results obtained in~\citep{a2021} have no implications for the type of questions considered in this paper, even in an analogous discrete setting.

\subsection{Contributions of this paper}

A main tool in several of the proofs in this paper is a current expansion for the charged model. Together with a coupling argument based on ideas from disagreement percolation, this provides a new approach to the Marcu--Fredenhagen ratio that does not involve cluster expansions. 
We also generalize a polymer expansion that first appeared in~\citep{os1978} to the multiple-charge setting \( k \geq 2\). Apart from the generalization, we also provide a fully rigorous treatment of this expansion, also in the case \( j=k=1 \).
All new ideas and tools can, with minor work, also be applied to the discrete setting, for, e.g., \( G= \mathbb{Z}_2,\) to obtain similar results for the analogs of the models considered here.

\subsection{Structure of paper}
 
 In Section~\ref{sec: preliminaries}, we introduce notation and describe the tools from discrete exterior calculus that we will use throughout the paper. 
 Next, in Section~\ref{sec: current expansion}, we describe a current expansion for a more general spin model, and then show that we obtain current expansions for \( U(1)\)-lattice gauge theory, the abelian lattice Higgs model, and the abelian lattice Higgs model with charge \( \geq 1 \) as special cases.
 In Section~\ref{section: new version of expansion}, we present a generalization of the polymer expansion of the partition function used in~\citep{os1978} for the case \( j = k=1\) to the case \( j,k \geq 1. \)
 Finally, in Section~\ref{section: proof of main result}, we give a proof of our main results, Theorem~\ref{theorem: main result charge nondiv} and Theorem~\ref{theorem: main result charge div}, from which Theorem~\ref{theorem: U1 Wilson line ratios MF} follows as an immediate corollary.

\subsection{Acknowledgements}
The author is grateful to Christophe Garban and Avelio Supelveda for insightful discussions. The author acknowledges support from the Swedish Research Council, grant number 2024-04744.

\section{Preliminaries}\label{sec: preliminaries}
\subsection{Notation and standing assumptions}

 To simplify the notation in the rest of this paper, we let
\begin{equation*}
	Z_{N,\beta,\kappa,k}[\gamma] \coloneqq \int  \rho(\sigma(\gamma)) e^{ \beta \sum_{p \in C_2(B_N)} \rho(d\sigma(p))+\kappa\sum_{e \in C_1(B_N)} \rho(\sigma(e))^k} \, d\mu(\sigma).
\end{equation*}

Throughout the paper, we will use notation from discrete exterior calculus. The list below summarizes this notation; for more careful definitions, we refer the reader to~\citep{flv2022}.

\begin{itemize}
	\item We let \( B_N \coloneqq [-N,N]^m \cap \mathbb{Z}^m.\)
	\item For \( k = 0,1,\dots, m, \) the set of oriented \( k \)-cells in \( B_N \) is denoted by \( C_k(B_N). \)

	\item For \( k = 0,1,\dots, m, \) a \(k \)-chain is formal sum of positively oriented \( k \)-cells with integer coefficients. The set of all \(k\)-chains on \( B_N \) is denoted by \( C^k(B_N,\mathbb{Z}) \)  
	
	\item For \( k=1,2,\dots,m  \) and \( c \in C_k(B_N), \) we let \( \partial c \in C^k(B_N,\mathbb{Z})\) denote the oriented boundary of \( c. \) For \( k = 0,1,\dots,m-1 \) and \( c \in C_k(B_N), \) we let \( \hat \partial c = \sum_{c' \in \partial c} c' \in C^{k+1}(B_N,\mathbb{Z}) \) denote the oriented coboundary of \( c. \) This notation extends to chains by linearity.
	
	\item For \( k = 1,2,\dots,m \) and \( C \subseteq C_k(B_N), \) we let \( \mathcal{G}(C) \) be the graph with vertex set \( C \) and an edge between \( c,c' \in C \) if and only if \( \support \partial c \cap \support  \partial c' \neq \emptyset.\) We say that \( C \) is connected if \( \mathcal{G}(C) \) is connected.

	\item A \( \{ -1,0,1\}\)-valued 1-chain \( \gamma \) with connected support such that \( \partial \gamma \in \{ -1,0,1 \}\) for all \( v \in C_0(B_N)^+ \)  and \(|\support \partial \gamma|\leq 2 \) is referred to as a path. A path \( \gamma \) with \( \partial \gamma = 0 \) is referred to as a loop. 
	
	\item For \( k = 0,1,\dots, m \), a \( G \)-valued function \( \omega \colon C_k(B_N) \to G \) with the property that \( {\omega(c) = -\omega(-c)} \) for all \( c \in C_k(B_N) \) is referred to as a \( k\)-form. The set of all \(k\)-forms is denoted by \( \Omega_k(B_N,G). \) When \( c \in C^k(B_N,\mathbb{Z}), \) we define \( \omega(c) = \sum_{c' \in C_k(B_N)^+} c[c']  \omega(c').\)
	
	\item For \( k=0,1,\dots, m-1, \) and \( \omega \in \Omega_k(B_N,G), \) we define \( d\omega \in \Omega_{k+1}(B_N,G)\) by
	\begin{equation*}
		d\omega(c) \coloneqq \omega(\partial c),\qquad c \in C_{k+1}(B_N) .
	\end{equation*}
	
	\end{itemize}

\section{Polymer expansions}

In this section, we present two different polymer expansions of the compact abelian lattice Higgs model with and without charge. 

\subsection{Current expansions}\label{sec: current expansion}

The main goal of this section is to construct a current expansion of the abelian lattice Higgs model with an integer charge, with \( U(1) \) lattice gauge theory and the abelian lattice Higgs model as special cases. To this end, we first construct a current expansion of a more general model. We then specialize this expansion to the three models we are interested in.
 %
%
To define the general spin model, let \( \ell \geq 0, \) and let \( \Gamma \subseteq C^\ell(B_N,\mathbb{Z})\) be symmetric, i.e., be such that if \( \xi \in \Gamma, \) then \( - \xi \in \Gamma. \) For \( \xi \in \Gamma, \) let \( \beta_\xi =-\beta_\xi >0, \) and consider the model on \( \Omega_\ell(B_N,U(1)) \) with Hamiltonian
\begin{equation}\label{eq: general model}
	H_{N,(\beta_\xi)}(\sigma) \coloneqq -\sum_{\xi \in \Gamma} \beta_\xi \rho(\sigma(\xi)), \quad \sigma \in \Omega_\ell(B_N,U(1)),
\end{equation}
using a uniform reference measure \( \mu. \)
This model describes a ferromagnetic spin system with block spins, which is invariant under spin-flips, and hence the Griffiths-Hurst-Sherman inequality holds. Moreover, the family of models described by~\eqref{eq: general model} has the following models as special cases.
\begin{enumerate}
	\item Letting \( \ell =1 \),  \( {\Gamma =   \{1\cdot e\colon e \in C_1(B_N)\},} \)  and \( \beta_\xi \equiv \kappa,  \) we obtain the XY model.
	\item Letting \( \ell =1 \), \( \Gamma = \bigl\{ \partial p \colon p\in C_2(B_N) \bigr\},  \) and \( \beta_\xi \equiv \beta,  \) we obtain \( U(1) \) lattice gauge theory.
	\item Letting \( \ell =1 \), \( {\Gamma = \bigl\{ \partial p \colon p\in C_2(B_N) \bigr\} \cup \{1\cdot e\colon e \in C_1(B_N)\},} \) and for \( \xi \in \Gamma, \) letting
	\[ \beta_\xi = \begin{cases}
		\beta &\text{if } \xi = \partial p \text{ for some }  p\in C_2(B_N) \cr \kappa &\text{else,}
	\end{cases}\] 
	we obtain the compact abelian lattice Higgs model (with charge \( k =1 \)).
	\item Letting \( \ell =1 \), \( {\Gamma = \bigl\{ \partial p \colon p\in C_2(B_N) \bigr\} \cup \{k\cdot e\colon e \in C_1(B_N)\},} \) and for \( \xi \in \Gamma, \) letting
	\[ \beta_\xi = \begin{cases}
		\beta &\text{if } \xi = \partial p \text{ for some }  p\in C_2(B_N) \cr \kappa &\text{else,}
	\end{cases}\] 
	we obtain the compact abelian lattice Higgs model with charge \( k. \)
\end{enumerate}

The following proposition gives a current expansion for the model described by~\eqref{eq: general model}.

\begin{proposition}\label{proposition: general currents}
	Let \( \ell \in \{ 0,1,\dots, m\}\), let \( \Gamma \subseteq C^\ell(B_N,\mathbb{Z})\) be symmetric, and let
	\begin{equation*}
		\mathcal{C} \coloneqq \bigl\{ \mathbf{n} \colon \Gamma\to \mathbb{N} \bigr\}.
	\end{equation*}
	For \( \gamma \in C^\ell(B_N,\mathbb{Z}),\) let 
	\begin{equation*}
		\mathcal{C}_{\gamma} \coloneqq \bigl\{ \mathbf{n} \in \mathcal{C} \colon (\mathbf{n}[\hat \partial c]-\mathbf{n}[-\hat \partial c]) + \gamma[c]= 0,  \; \forall c\in C_\ell(B_N) \bigr\},
	\end{equation*}
	where, for \( c \in C_\ell(B_N), \) we let 
	\[
		\mathbf{n}[\hat \partial c] \coloneqq \sum_{\xi \in \Gamma } \xi[c] \cdot \mathbf{n}[\xi].
	\]
	For each \( \xi \in \Gamma, \) let \( \beta_\xi=\beta_{-\xi} >0, \) and for \( \mathbf{n} \in \mathcal{C}_{\gamma}, \) let
	\begin{equation*}
		w(\mathbf{n}) \coloneqq \prod_{\xi \in \Gamma } \frac{ \beta_\xi^{\mathbf{n}[\xi]}}{\mathbf{n}[\xi]!} .
	\end{equation*}
	Finally, let \( \mu \) be the uniform distribution on \( \Omega_\ell(B_N,U(1)). \) Then
    \begin{equation*}
        \begin{split}
        & \int \rho(\sigma(\gamma)) e^{  \sum_{\xi \in \Gamma} \beta_\xi \rho(\sigma(\xi))} \, d\mu(\sigma)
        =
        \sum_{n \in \mathcal{C}_{\gamma}} w(n) .
        \end{split}
    \end{equation*}
\end{proposition}

\begin{proof}
    First note that, for any \( \sigma \in \Omega_\ell(B_N,U(1)), \) we have
    \begin{equation*}
        \begin{split}
            & e^{ \sum_{\xi \in \Gamma} \beta_\xi \rho(\sigma(\xi))}
            = \prod_{\xi \in \Gamma} e^{ \beta_\xi   \rho(\sigma(\xi))}
            = \prod_{\xi \in \Gamma} \sum_{\mathbf{n}[\xi] \geq 0} \frac{\bigl( \beta_\xi   \rho(\sigma(\xi)) \bigr)^{\mathbf{n}[\xi]}}{\mathbf{n}[\xi]!}.
        \end{split}
    \end{equation*}
    This implies, in particular, that
    \begin{equation*}
        \begin{split}
            & \int  \rho(\sigma(\gamma)) e^{  \sum_{\xi \in \Gamma} \beta_\xi \rho(\sigma(\xi))} \, d\mu(\sigma)
            = 
            \int  \rho(\sigma(\gamma))\prod_{\xi \in \Gamma} \sum_{\mathbf{n}[\xi] \geq 0} \frac{\bigl( \beta_\xi   \rho(\sigma(\xi)) \bigr)^{\mathbf{n}[\xi]}}{\mathbf{n}[\xi]!} \, d\mu(\sigma)
            \\&\qquad = 
            \int  
            \rho(\sigma(\gamma)) 
            \sum_{\mathbf{n} \in \mathcal{C}}
            \prod_{\xi \in \Gamma} \frac{\bigl( \beta_\xi   \rho(\sigma(\xi)) \bigr)^{\mathbf{n}[\xi]}}{\mathbf{n}[\xi]!} \, d\mu(\sigma)
            .
        \end{split}
    \end{equation*} 
    Now fix \( c_0 \in C_\ell(B_N) \), \( \sigma_0 \in \Omega_\ell(B_N, U(1)), \) and let \( \mu_0 \) be the restriction of \( \mu \) to \( \pm c_0 \) after conditioning \( \sigma \in \Omega_\ell(B_N, U(1))\) to be equal to \( \sigma_0 \) outside of \( \pm c_0. \)
    Then, for any \( \mathbf{n} \in \mathcal{C},\) we have
    \begin{equation*}
        \begin{split}
            &
            \int \rho(\sigma(\gamma))\prod_{\xi \in \Gamma} \frac{\bigl( \beta_\xi   \rho(\sigma(\xi)) \bigr)^{\mathbf{n}[\xi]}}{\mathbf{n}[\xi]!}\, d\mu_0(\sigma )
            \\&\qquad=
            \prod_{c \in \gamma }\rho(\sigma_0(c))\prod_{\xi \in \Gamma } \frac{\bigl( \beta_\xi   \rho(\sigma_0(\xi)) \bigr)^{\mathbf{n}[\xi]}}{\mathbf{n}[\xi]!} 
            \\&\qquad\qquad \cdot\bigl(  \rho(\sigma_0(c_0)) \bigr)^{-  (\mathbf{n}[\hat \partial c_0]-\mathbf{n}[-\hat \partial c_0]) -\gamma[c_0]}
            \int  \bigl(    \rho(\sigma(c_0)) \bigr)^{ (\mathbf{n}[\hat \partial c_0]-\mathbf{n}[-\hat \partial c_0]) + \gamma[c_0]}   \, d\mu_0(\sigma)  .
        \end{split}
    \end{equation*} 
    Since
    \begin{align*}
    	\int  \bigl(    \rho(\sigma(c_0)) \bigr)^{ (\mathbf{n}[\hat \partial c_0]-\mathbf{n}[-\hat \partial c_0]) + \gamma[c_0]}   \, d\mu_0(\sigma) = \mathbf{1}\bigl( (\mathbf{n}[\hat \partial c_0]-\mathbf{n}[-\hat \partial c_0]) + \gamma[c_0]= 0 \bigr),
    \end{align*}
    it follows that,
    \begin{equation*}
        \begin{split}
            & \int  \rho(\sigma(\xi)) e^{ \beta_\xi \sum_{p \in C_2(B_N)} \rho(d\sigma(p))} \, d\mu(\sigma)
            \\&\qquad = 
            \sum_{\mathbf{n}\in \mathcal{C}} \prod_{c \in \gamma }\rho(\sigma(c))\prod_{\xi \in \Gamma } \frac{\bigl( \beta_\xi   \rho(\sigma(\xi)) \bigr)^{\mathbf{n}[\xi]}}{\mathbf{n}[\xi]!} \cdot \mathbf{1}\bigl( \forall c \in C_\ell(B_N)^+ \colon (\mathbf{n}[\hat \partial c]-\mathbf{n}[-\hat \partial c]) + \gamma[c]= 0 \bigr)
            \\&\qquad=
            \sum_{\mathbf{n}\in \mathcal{C}} \prod_{\xi \in \Gamma} 
            \frac{\beta_\xi^{\mathbf{n}[\xi]}}{\mathbf{n}[\xi]!}
            \prod_{c \in C_\ell(B_N)^+} \rho(\sigma(c))^{(\mathbf{n}[\hat \partial c]-\mathbf{n}[-\hat \partial c])+\gamma[c]} \cdot \mathbf{1}\bigl( \forall c \in C_\ell(B_N)^+ \colon (\mathbf{n}[\hat \partial c]-\mathbf{n}[-\hat \partial c]) + \gamma[c]= 0 \bigr)
            \\&\qquad=
            \sum_{\mathbf{n}\in \mathcal{C}} \prod_{\xi \in \Gamma} 
            \frac{\beta_\xi^{\mathbf{n}[\xi]}}{\mathbf{n}[\xi]!}
            \cdot \mathbf{1}\bigl( \forall c \in C_\ell(B_N)^+ \colon (\mathbf{n}[\hat \partial c]-\mathbf{n}[-\hat \partial c]) + \gamma[c]= 0 \bigr)
            =
            \sum_{\mathbf{n} \in \mathcal{C}_\xi} \prod_{\xi \in \Gamma  } \frac{ \beta_\xi^{\mathbf{n}[\xi]}}{\mathbf{n}[\xi]!} .
        \end{split}
    \end{equation*}
    This concludes the proof.
\end{proof}

We next describe the special case of Proposition~\ref{proposition: general currents} that will be most relevant to us. In other words, in the proposition below we obtain a current expansion of the compact abelian lattice Higgs model with charge \( k \geq 2. \)

 \begin{proposition} \label{proposition: current for Higgs with charge}
	Let \( \mu \) be the uniform distribution on \( \Omega_2(B_N,U(1)), \) and let
	\begin{equation*}
		\mathcal{C} \coloneqq \bigl\{ \mathbf{n} \colon C_1(B_N)\cup C_2(B_N) \to \mathbb{N} \bigr\}.
	\end{equation*}
	Let \( \gamma \) be a path, let \( k \geq 0, \) and let
	\begin{equation*}
		\mathcal{C}_{\gamma,k} \coloneqq \bigl\{ \mathbf{n} \in \mathcal{C} \colon (\mathbf{n}[\hat \partial e]-\mathbf{n}[-\hat \partial e]) + k(\mathbf{n}[e]-\mathbf{n}[-e]) +  \gamma[e]= 0,  \; \forall e\in C_1(B_N) \bigr\}.
	\end{equation*}
	Further, let \( \beta,\kappa \geq 0, \) and for \( \mathbf{n} \in \mathcal{C}_{\gamma,k}, \) let
	\begin{equation*}
		w_{\beta,\kappa}(\mathbf{n}) \coloneqq \prod_{p \in C_2(B_N) } \frac{ \beta^{\mathbf{n}[p]}}{\mathbf{n}[p]!} \prod_{e \in C_1(B_N) } \frac{ \kappa^{\mathbf{n}[e]}}{\mathbf{n}[e]!} .
	\end{equation*}
	Then
    \begin{equation*} 
        \begin{split}
        & \int \rho(\sigma(\gamma)) e^{ \beta \sum_{p \in C_2(B_N)} \rho(d\sigma(p)) + \kappa \sum_{e \in C_1(B_N)} \rho(\sigma(e))^k} \, d\mu(\sigma)
        =
        \sum_{n \in \mathcal{C}_{\gamma,k}} w_{\beta,\kappa}(n) .
        \end{split}
    \end{equation*} 
\end{proposition}

\begin{proof}
	Letting \( \ell =1 \), \( {\Gamma = \bigl\{ \partial p \colon p\in C_2(B_N) \bigr\} \cup \{j\cdot e\colon e \in C_1(B_N)\},} \) and for \( \gamma \in \Gamma, \) letting
	\[ \beta_\gamma = \begin{cases}
		\beta &\text{if } \gamma = \partial p \text{ for some } \in p\in C_2(B_N) \cr \kappa &\text{else,}
	\end{cases}\] 
	the desired conclusion follows immediately from Proposition~\ref{proposition: general currents}.
\end{proof}

For a path \( \gamma \) and in integer \( k \geq 1, \) we let \( \mathbf{P}^{\gamma,k} \) be the measure on \( \mathcal{C}_{\gamma,k} \) induced by \( w_{\beta,\kappa}, \) i.e., the measure defined by
 \begin{equation*}
 	\mathbf{P}^{\gamma,k}(\mathbf{n}) \coloneqq \mathbf{P}^{\gamma,k}_{\beta,\kappa}(\mathbf{n}) \coloneqq  \frac{w_{\beta,\kappa}(\mathbf{n})}{\sum_{\mathbf{m} \in \mathcal{C}_{\gamma,k}} w_{\beta,\kappa}(\mathbf{m)}},\quad \mathbf{n} \in \mathcal{C}_{\gamma,k}.
 \end{equation*}

\subsection{A polymer expansion for the Higgs/confinement phase}\label{section: new version of expansion}

In this section, we present an expansion of \( Z_{j,k}[\gamma] \) which can be used as a polymer expansion when either \( \beta \) is sufficiently small or \( \kappa \) is sufficiently large.

\subsubsection{Charge 1}\label{section: charge 1 expansion}

In this section, we consider the case \( k = 1 \) and rewrite \( Z_{j,1}[\gamma] \) as a sum of clusters. The main idea behind this expansion originates from~\citep[Section 3, Section 4]{os1978} (see also~\citep[Appendix]{fs1979}), and is based on the observation that if \( \beta \) is small \emph{or} \( \kappa \) is large, the term \( e^{\beta ( \Re  \rho(d\sigma(p)) -1)} -1 \) in \( Z_{j,k}[\gamma] \) will typically be small.

Recall that we say that a set \( C \subseteq C_k(B_N)^+ \) is connected if the graph \( \mathcal{G}(C) \) is connected, where \( \mathcal{G}(C) \) is the graph  with vertex set \( C \) and an edge between \( c_1,c_2 \in C \) if \( \support \partial c_1 \cap \support \partial c_2 \neq \emptyset. \) With this notation, we note that each \( P \subseteq C_2(B_N)^+ \) induces a partition \( \mathcal{P}_P \) of \( P  \) into connected sets. For \( P \subseteq C_2(B_N)^+ \) and \( {e \in C_1(B_N)^+,} \) we write \( e \sim P \) if there is \( p \in P \) such that \( e \in \support \partial p. \)

\color{black}

\begin{proposition}\label{proposition: polymer expansion}
	Let \( \beta,\kappa \geq 0, \)  let \( j \geq 1,\) and let \( \gamma \) be a path. Then
	\begin{equation}\label{eq: polymer expansion}
		Z_{j,1}[\gamma]
		=
		e^{\beta |C_2(B_N)|} b_0^{|C_1(B_N)^+| } (b_j/b_0)^{|\gamma|} 
		\sum_{P \subseteq C_2(B_N)^+}  
		\prod_{P' \in \mathcal{P}_P}  \phi_\gamma(P') ,
	\end{equation} 
	where, for \( i \geq 0, \) we have
	\begin{align*}
		b_i \coloneqq \int  \rho(\sigma(e))^i e^{ 2\kappa   \Re \rho(\sigma(e)) } \, d\mu(\sigma)
	\end{align*} 
	and 
	\begin{align*}
		&\phi_\gamma(P') \coloneqq   (b_0 /b_j)^{|\{ e \in \gamma \colon e \sim P' \}|} 
		\int \prod_{e \in \gamma \colon e \sim P'}\rho(\sigma(e))^j  \prod_{p \in P'}
		\bigl( e^{\beta ( \Re  \rho(d\sigma(p)) -1)} -1\bigr) \prod_{e \in C_1(B_N)^+}\frac{ e^{ 2\kappa  \Re \rho(\sigma(e)) } \, d\mu(\sigma)}{b_0}.
	\end{align*}
\end{proposition}

\begin{remark}
	The constants \( (b_i)_{i \geq 0} \) appearing in Proposition~\ref{proposition: polymer expansion} can be calculated explicitly, as \( b_i = I_i(2\kappa), \) where \( z \mapsto I_i(z)\) is the modified Bessel function of the first kind with index \( i \).
\end{remark}

\begin{proof}[Proof of Proposition~\ref{proposition: polymer expansion}]
	First, note that for any \( \sigma \in \Omega_1(B_N,U(1)) \), we can write
	\begin{align*}
		&e^{\beta \sum_{p\in C_2(B_N)} \rho(d\sigma(p))  }
		=
		e^{2\beta \sum_{p\in C_2(B_N)^+} \Re \rho(d\sigma(p))  }
		=
		\prod_{p\in C_2(B_N)^+}
		e^{2\beta  \Re \rho(d\sigma(p)) }
		\\&\qquad =
		e^{\beta |C_2(B_N)|} \prod_{p\in C_2(B_N)^+}
		\bigl( 1+e^{\beta ( \Re \rho(d\sigma(p)) -1)} -1\bigr) 
		\\&\qquad =
		e^{2\beta |C_2(B_N)|}  \sum_{P \subseteq C_2(B_N)^+ } \prod_{p \in P}
		\bigl( e^{2\beta ( \Re  \rho(d\sigma(p)) -1)} -1\bigr) ,
	\end{align*}
	and hence
	\begin{align*}
		&\rho(\sigma(\gamma))^j e^{\beta \sum_{p\in C_2(B_N)} \rho(d\sigma(p)) + \kappa \sum_{e  \in C_1(B_N)} \rho(\sigma(e)) }
		\\&\qquad =
		e^{2\beta |C_2(B_N)|} \prod_{e \in \gamma}\rho(\sigma(e))^j  \sum_{P \subseteq C_2(B_N)^+ } \prod_{p \in P}
		\bigl( e^{2\beta ( \Re  \rho(d\sigma(p)) -1)} -1\bigr) e^{ 2\kappa \sum_{e  \in C_1(B_N)^+} \Re \rho(\sigma(e)) }.
	\end{align*}
	Now note that for any \( P \subseteq C_2(B_N)^+ \) we have
	\begin{align*}
		&\int \prod_{e \in \gamma}\rho(\sigma(e))^j  \prod_{p \in P}
		\bigl( e^{2\beta ( \Re  \rho(d\sigma(p)) -1)} -1\bigr) e^{ 2\kappa \sum_{e  \in C_1(B_N)^+} \Re \rho(\sigma(e)) } \, d\mu(\sigma)
		\\&\qquad=
		b_0^{|C_1(B_N)^+|}\int \prod_{e \in \gamma}\rho(\sigma(e))^j  \prod_{p \in P}
		\bigl( e^{2\beta ( \Re  \rho(d\sigma(p)) -1)} -1\bigr) \prod_{e \in C_1(B_N)^+} \frac{e^{ 2\kappa  \Re \rho(\sigma(e)) }}{b_0} \, d\mu(\sigma)
		\\&\qquad 
		= b_0^{|C_1(B_N)^+|}  (b_j/b_0)^{|\{ e \in \gamma \colon e \nsim P \}|} \prod_{P' \in \mathcal{P}_P} \hat \phi_\gamma(P') ,
	\end{align*} 
	where 
	\begin{align*}
		\hat \phi_\gamma(P') \coloneqq\int \prod_{e \in \gamma \colon e \sim P'}\rho(\sigma(e))^j  \prod_{p \in P'}
		\bigl( e^{2\beta ( \Re  \rho(d\sigma(p)) -1)} -1\bigr) \frac{e^{ 2\kappa  \Re \rho(\sigma(e)) }}{b_0}  \, d\mu(\sigma).
	\end{align*} 
	Finally, noting that
	\begin{align*}
		&\phi_\gamma(P') = 
		\frac{\hat \phi_\gamma(P')}{ (b_j/b_0)^{|\{ e \in \gamma \colon e \sim P' \}|}},
	\end{align*}
	we obtain
	\begin{align*}
		&Z_{j,1}[\gamma]
		=
		\int \rho(\sigma(\gamma))^j e^{\beta \sum_{p\in C_2(B_N)} \rho(d\sigma(p)) + \kappa \sum_{e  \in C_1(B_N)} \rho(\sigma(e)) }\, d\mu(\sigma)
		\\&\qquad=
		e^{\beta |C_2(B_N)|} b_0^{|C_1(B_N)^+| } (b_j/b_0)^{|\gamma|} 
		\sum_{P \subseteq C_2(B_N)^+}  
		\prod_{P' \in \mathcal{P}_P}  \phi_\gamma(P').
	\end{align*}
	This concludes the proof.
\end{proof}

In order for~\eqref{eq: polymer expansion} to be useful as a polymer expansion, we need to show that the polymers appearing in this expansion are typically small. This is the purpose of the next lemma.

\begin{lemma}\label{lemma: upper Holder bound 1}
	Let \( \beta,\kappa \geq 0, \) let \( j \geq 1, \) and let \( p_0 \in C_2(B_N)^+. \) There is a constant \( a_m \) that only depends on \( m \) such that for any connected set \( P \subseteq C_2(B_N)^+, \) we have
	\begin{equation}\label{eq: upper Holder bound 1}
		\bigl| \phi_\gamma(P) \bigr|
		\leq 
		 (b_0/b_j)^{|\{ e \in \gamma \colon  e \sim P \} |} 
		\Bigl( \int 
		\bigl| e^{2\beta ( \Re  \rho(d\sigma(p_0)) -1)} -1\bigr|^{a_m}
		\prod_{\substack{e \in \support\partial p_0 
		}} \frac{
		e^{ 2\kappa \Re \rho(\sigma(e)) } }{ b_{0} } 
		\, d\mu(\sigma)
		\Bigr)^{|P|/a_m}.
	\end{equation}
	In particular, for any \( \kappa >0 \), we have
	\begin{align*}
		&\sup_{P \subseteq C_2(B_N)^+}\bigl|\phi_\gamma(P)\bigr|^{1/|P|} \leq
		(b_{0}/b_{j})^4 (1-e^{-4\beta})   \searrow 0 \quad \text{as } \beta \to 0,
	\end{align*}
	and for any \( \beta>0, \) we have
	\begin{align*}
		&\lim_{\kappa \to \infty} \sup_{P \subseteq C_2(B_N)^+}\bigl| \phi_\gamma(P) \bigr|^{1/|P|} 
		\leq C_{\beta,\kappa} \searrow 0 \quad \text{as } \kappa \to \infty.
	\end{align*}
\end{lemma}

 \begin{remark}
 	In Figure~\ref{fig: expansionweight}, we draw the level sets of the right-hand side of~\eqref{eq: upper Holder bound 1} for \( a_m = 6. \) In particular, we note that to ensure that this upper bound is small when \( \kappa \) is small, \( \beta \) needs to be small \emph{compared to} \( \kappa. \) This further motivates complementing the expansion described in~Proposition{proposition: polymer expansion} with the current expansion described in~Proposition~\ref{proposition: current for Higgs with charge}, as there the corresponding upper bound can be bounded from above uniformly in \( \beta \) when \( \kappa \) is small.
 \end{remark}

\begin{proof}[Proof of Lemma~\ref{lemma: upper Holder bound 1}]
	Let \( P \subseteq C_2(B_N). \) Then \( P \) can be partitioned into at most \( a_m \) sets \( P_1, \dots, P_{a_m} \) such that for each set, no two plaquettes have a common edge in their boundary. Fix such a partition \( (P_i)_{i \in I} \). Then 
	\begin{align*}
		&\bigl| \phi_\gamma(P) \bigr|
		= 
		(b_0 /b_j)^{|\{ e \in \gamma \colon e \sim P \}|}  
		\int 
		\prod_{p \in P}
		\bigl| e^{2\beta ( \Re  \rho(d\sigma(p)) -1)} -1\bigr|
		\prod_{\substack{e \in C_1(B_N)^+ \colon \\ e \sim P
		}} \frac{
		e^{ 2\kappa \Re \rho(\sigma(e)) } }{ b_0 }
		\, d\mu(\sigma)
		\\&\qquad
		= 
		(b_0 /b_j)^{|\{ e \in \gamma \colon e \sim P \}|}
		\int 
		\prod_{i \in I} \prod_{p \in P_i}
		\bigl| e^{2\beta ( \Re  \rho(d\sigma(p)) -1)} -1\bigr|
		\prod_{\substack{e \in C_1(B_N)^+ \colon \\ e \sim P
		}} \frac{
		e^{ 2\kappa \Re \rho(\sigma(e)) } }{ b_0 }
		\, d\mu(\sigma)
		.
	\end{align*}
	Now note that the measure 
	\[
	 \prod_{\substack{e \in C_1(B_N)^+ \colon \\ e \sim P
		}} \frac{
		e^{ 2\kappa \Re \rho(\sigma(e)) } }{ b_0 }
		\, d\mu(\sigma) 
	\]
	is a probability measure. Applying H\"olders inequality with respect to this  measure, it follows that
	\begin{align*}
		&  
		\int 
		\prod_{i \in I} \prod_{p \in P_i}
		\bigl| e^{2\beta ( \Re  \rho(d\sigma(p)) -1)} -1\bigr|
		\prod_{\substack{e \in C_1(B_N)^+ \colon \\ e \sim P
		}} \frac{
		e^{ 2\kappa \Re \rho(\sigma(e)) } }{ b_0 }
		\, d\mu(\sigma)
		\\&\qquad\leq  
		\prod_{i \in I}
		\Bigl( \int 
		\prod_{p \in P_i}
		\bigl| e^{2\beta ( \Re  \rho(d\sigma(p)) -1)} -1\bigr|^{|I|} 
		\prod_{\substack{e \in C_1(B_N)^+ \colon \\ e \sim P
		}} \frac{
		e^{ 2\kappa \Re \rho(\sigma(e)) } }{ b_{0} }
		\, d\mu(\sigma)
		\Bigr)^{1/|I|}.
		\end{align*}
		Since the plaquettes in each \( P_i \) have no common edges in their boundary, it follows that the right-hand side of the previous equation is equal to 
		\begin{align*}
		&
		\prod_{i \in I} \prod_{p \in P_i}
		\Bigl( \int 
		\bigl| e^{2\beta ( \Re  \rho(d\sigma(p)) -1)} -1\bigr|^{|I|} 
		\prod_{\substack{e \in C_1(B_N)^+ \colon \\ e \sim P
		}} \frac{
		e^{ 2\kappa \Re \rho(\sigma(e)) } }{ b_{0} }
		\, d\mu(\sigma)
		\Bigr)^{1/|I|}
		\\&\qquad= 
		\prod_{i \in I} \prod_{p \in P_i}
		\Bigl( \int 
		\bigl| e^{2\beta ( \Re  \rho(d\sigma(p)) -1)} -1\bigr|^{|I|} 
		\prod_{\substack{e \in \partial p
		}} \frac{
		e^{ 2\kappa \Re \rho(\sigma(e)) } }{ b_{0} }
		\, d\mu(\sigma)
		\Bigr)^{1/|I|}
		\\&\qquad= 
		\Bigl( \int 
		\bigl| e^{2\beta ( \Re  \rho(d\sigma(p_0)) -1)} -1\bigr|^{|I|} 
		\prod_{\substack{e \in \partial p_0
		}} \frac{
		e^{ 2\kappa \Re \rho(\sigma(e)) } }{ b_{0} }
		\, d\mu(\sigma)
		\Bigr)^{|P|/|I|}.
	\end{align*} 
	Combining the above equations, since \( |I| \leq a_m \), the desired conclusion follows. 
\end{proof}

\subsubsection{Charge \( \geq 2 \)}\label{section: charge k expansion}

In this section, we extend the ideas of Section~\ref{section: charge 1 expansion} to the case \( k \geq 2, \) i.e., to the compact abelian Higgs model with charge \( k \geq 2. \)

The following observation will be fundamental for this section. Assume that \( k \geq 1 \) and let \( U(1)_k \) be the restriction of \( U(1) \) to angles in \( (-\pi/k,\pi/k). \) Then, for any \( g \in U(1)_k \) there are unique \( h \in U(1) \) and  \( h' \in \mathbb{Z}_k = \{ g \in U(1)\colon g^k = 1\} \) such that
	\begin{equation*}
		\rho(g) = \rho(h)\rho(h'),
	\end{equation*}
	and hence
	\begin{equation*}
		\rho(g)^k = \rho(h)^k\rho(h')^k = \rho(h)^k.
	\end{equation*} 
	Using this notation, we can write
	\begin{equation}\label{eq: new notation for angles}\begin{split}
		 &Z_{j,k}[\gamma] = \int \rho(\sigma(\gamma))^j e^{\beta \sum_{p\in C_2(B_N)} \rho(d\sigma(p)) + \kappa \sum_{e  \in C_1(B_N)} \rho(\sigma(e))^k } \, d\mu(\sigma)
		 \\&\qquad=
		 \int \rho(\theta(\gamma))^j \rho(\theta'(\gamma))^j   e^{\beta \sum_{p\in C_2(B_N)} \rho(\theta(\partial p)) \rho(\theta'(\partial p)) + \kappa \sum_{e  \in C_1(B_N)} \rho(\theta)^k} \, d\mu(\theta, \theta'),
	\end{split}\end{equation}
	where \( \mu(\theta,\theta ') \) is the uniform measure on \( \Omega_1(B_N,U(1)_k \times \mathbb{Z}_k)\).

Given \( P \subseteq C_2(B_N)^+\) and \( \theta' \in \Omega_1(B_N,\mathbb{Z}_k), \) we say that \( (P,\theta') \) is \emph{connected} if \( \mathcal{G}(P \cap (\support d\theta')^+) \)  is connected. For \( P \subseteq C_2(B_N)^+ \) and \( \theta' \subseteq \Omega_1(B_N,\mathbb{Z}_k) \), let  \( \mathcal{P}_{P, \theta'} \) be the set of all maximal connected subsets of  \( (P,\hat \theta') .\)

\begin{proposition}\label{proposition: expansion 2}
	Let \( \beta,\kappa \geq 0, \) let \(j \geq 0 \) and \(k \geq 1,\) and let \( e_0 \in C_1(B_N)^+. \) Then
	
	\begin{align*}
		 Z_{j,k}[\gamma] 
		 &= e^{\beta|C_2(B_N)|} k^{-|C_1(B_N)|} b_{j,k}^{|\gamma|}
		 \sum_{P,\hat \theta}
		 \prod_{(P',\hat \theta') \in \mathcal{P}_{P,\hat \theta}}
		 \phi(P',\hat \theta'),
	\end{align*}
	where for \( i \geq 0 \) we have
	\begin{align*}
		b_{i,k} \coloneqq \int \rho(\theta(e_0))^i e^{2\kappa \rho(\theta(e_0))^{k}} \, d\mu(\theta,\theta')
	\end{align*}
	and
	\begin{align*}
		\phi(P',\hat \theta') &\coloneqq 
		\prod_{\substack{e \in \gamma \colon\\ e \sim (P',\hat \theta')}} (b_{0,k}/b_{j,k})  \prod_{p\in (\support \hat \theta')^+ } e^{2\beta (\Re \rho(\hat \theta'(\partial p)) -1)}
		\\&\qquad \cdot \int \prod_{\substack{e \in C_1(B_N)^+ \colon \\ e \sim (P',\hat \theta')}} \rho(\theta(e))^j  
		 \prod_{p \in (P',\hat \theta')} \bigl( e^{2\beta \Re \rho(\theta'(\partial p)) (\rho(\theta(\partial p)) -1)}-1\bigr) \prod_{e \in C_1(B_N)^+}\frac{e^{2\kappa \Re \rho(\theta(e))^k  }}{b_{0,k}} \, d\mu(\theta) .
	\end{align*} 
\end{proposition}

\begin{remark}
	When \( k = 1, \) the expansion in Proposition~\ref{proposition: expansion 2} is identical to the expansion in Proposition~\ref{proposition: polymer expansion}.
\end{remark}

\begin{remark}
	One verifies that
	\begin{align*}
		b_{0,k} = \frac{1}{2\pi/k}\int_{-\pi/k}^{\pi/k}  e^{2\kappa \Re e^{ik\theta}} \, d\theta = [t = k\theta,\, dt = kd\theta] = \frac{1}{2\pi} \int_{-\pi}^{\pi}  e^{2\kappa \Re e^{it}} \, dt =  I_0(2\kappa)
	\end{align*}
	and 
	\begin{align*}
		b_{k,k} = \frac{1}{2\pi/k}\int_{-\pi/k}^{\pi/k} e^{ik\theta} e^{2\kappa \Re e^{ik\theta}} \, d\mu(\theta) = [t=k\theta] = \frac{1}{2\pi} \int_{-\pi}^{\pi} e^{it} e^{2\kappa \Re e^{it}} \, d\mu(t) = I_1(2\kappa),
	\end{align*}
	where for \( i \geq 0, \) \( z \mapsto I_i(z)\) is the modified Bessel function of the first kind with index \( i . \)
\end{remark}

\begin{proof}[Proof of Proposition~\ref{proposition: expansion 2}]
	Note first that by~\eqref{eq: new notation for angles}, we have 
	\begin{equation*} \begin{split}
		 &Z_{j,k}[\gamma] = \int \rho(\theta(\gamma))^j \rho(\theta'(\gamma))^j   e^{\beta \sum_{p\in C_2(B_N)} \rho(\theta(\partial p)) \rho(\theta'(\partial p)) + \kappa \sum_{e  \in C_1(B_N)} \rho(\theta)^k} \, d\mu(\theta, \theta').
		\end{split}
	\end{equation*}
	We now rewrite the terms of this expression as follows. 
	First, write
	\begin{equation}\label{eq: first step}
		e^{\beta \sum_{p\in C_2(B_N)} \rho(\theta(\partial p)) \rho(\theta'(\partial p))}
		=
		e^{\beta \sum_{p\in C_2(B_N)} \rho(\theta'(\partial p)) }
		e^{\beta \sum_{p\in C_2(B_N)} \rho(\theta'(\partial p))(\rho(\theta(\partial p)) -1)} .
	\end{equation}
	For the first term in the right-hand side of~\eqref{eq: first step}, we have 
	\begin{align*}
		&e^{\beta \sum_{p\in C_2(B_N)} \rho(\theta'(\partial p)) }
		=
		\prod_{p\in C_2(B_N)^+} e^{2\beta    \Re \rho(\theta'(\partial p)) }
		=
		e^{\beta|C_2(B_N)|}
		\prod_{p\in (\support \hat \theta)^+}  e^{2\beta  ( \Re \rho(\theta'(\partial p))-1) },
	\end{align*} 
	and for the second term on the right-hand side of~\eqref{eq: first step}, we have
	\begin{align*}
		&e^{\beta \sum_{p\in C_2(B_N)} \rho(\theta'(\partial p))(\rho(\theta(\partial p)) -1)} 
		=
		e^{2\beta \sum_{p\in C_2(B_N)^+} \Re \rho(\theta'(\partial p))(\rho(\theta(\partial p)) -1)} 
		\\&\qquad=
		\prod_{p \in C_2(B_N)^+} e^{2\beta  \Re \rho(\theta'(\partial p))(\rho(\theta(\partial p)) -1)} 
		=
		\prod_{p \in C_2(B_N)^+} \bigl( e^{2\beta  \Re \rho(\theta'(\partial p))(\rho(\theta(\partial p)) -1)} -1+1 \bigr)
		\\&\qquad=
		\sum_{P \subseteq C_2(B_N)^+} \prod_{p \in P} \bigl( e^{2\beta  \Re \rho(\theta'(\partial p))(\rho(\theta(\partial p)) -1)} -1 \bigr).
	\end{align*}
	Combining the above expressions, we obtain
	\begin{align*}
		 Z_{j,k}[\gamma] 
		 &= e^{\beta|C_2(B_N)|}
		 \sum_{P \subseteq C_2(B_N)^+}
		 \int 
		 \rho(\theta(\gamma))^j \rho(\theta'(\gamma))^j  
		  \prod_{p \in P} \bigl( e^{2\beta  \Re \rho(\theta'(\partial p))(\rho(\theta(\partial p)) -1)} -1 \bigr) 
		 \\&\qquad\qquad\qquad\qquad\qquad\qquad\cdot  \prod_{p\in (\support \hat \theta)^+}  e^{2\beta  ( \Re \rho(\theta'(\partial p))-1) }
		 e^{\kappa \sum_{e  \in C_1(B_N)} e^{ik\theta[e]} } \, d\mu( \theta,\theta')
		 \\
		 \\&= e^{\beta|C_2(B_N)|} k^{-|C_1(B_N)^+|}
		 \sum_{\substack{P \subseteq C_2(B_N)^+\\ \theta' \in \Omega_1(B_N,\mathbb{Z}_k)}}
		 \int 
		 \rho(\theta(\gamma))^j \rho(\theta'(\gamma))^j  
		  \prod_{p \in P} \bigl( e^{2\beta  \Re \rho(\theta'(\partial p))(\rho(\theta(\partial p)) -1)} -1 \bigr) 
		 \\&\qquad\qquad\qquad\qquad\qquad\qquad\qquad\qquad\cdot  \prod_{p\in (\support \hat \theta)^+}  e^{2\beta  ( \Re \rho(\theta'(\partial p))-1) }
		 e^{\kappa \sum_{e  \in C_1(B_N)} e^{ik\theta[e]} } \, d\mu(\theta).
	\\&=
		 e^{\beta|C_2(B_N)|} k^{-|C_1(B_N)^+|} b_{0,k}^{|C_1(B_N)^+|}
		 \sum_{\substack{P \subseteq C_2(B_N)^+\\ \theta' \in \Omega_1(B_N,\mathbb{Z}_k)}}
		 \prod_{e \in \gamma \colon e \nsim (P,\hat \theta)} (b_{j,k}/b_{0,k})
		 \prod_{(\hat P,\hat \theta') \in \mathcal{P}_{P, \theta'}}
		 \hat \phi(\hat P,\hat \theta')
		 \\&\qquad=
		 e^{\beta|C_2(B_N)|} k^{-|C_1(B_N)|}
		 b_{0,k}^{|C_1(B_N)^+|}
		 b_{j,k}^{|\gamma|}
		  \sum_{\substack{P \subseteq C_2(B_N)^+\\ \theta' \in \Omega_1(B_N,\mathbb{Z}_k)}}
		 \prod_{e \in \gamma \colon e \sim (P, \theta')} (b_{0,k}/b_{j,k})
		 \prod_{(\hat  P,\hat \theta') \in \mathcal{P}_{P,\hat \theta}}
		 \hat \phi(\hat P,\hat \theta'),
	\end{align*}
	where
	\begin{align*}
		\hat \phi(\hat P,\hat \theta') \coloneqq &
		\prod_{p\in (\support \hat \theta')^+}e^{2\beta (\Re \rho(\hat \theta'(\partial p))-1)}
		\\&\qquad\cdot \int \prod_{e \sim (\hat P,\hat \theta')} \rho(\hat \theta'(e))^j\rho(\theta(e))^j  
		 \prod_{p \in (\hat  P,\hat \theta')} \bigl(e^{2\beta \Re  \rho(\hat \theta'(\partial p) ) ( \rho(\theta(\partial p)) -1)}-1\bigr) 
		 \frac{e^{\kappa \sum_{e  \in C_1(B_N)} \rho(\theta(e))^k  } }{b_{0,k}^{|C_1(B_N)^+|}}\, d\mu(\theta).
	\end{align*}  
	From this, the desired conclusion immediately follows.
\end{proof}

The next result, Lemma~\ref{lemma: upper bound of expansion} below, gives an upper bound of \( \bigl| \phi_\gamma(P, \theta') \bigr| \) which is needed for the polymer expansion described in Proposition~\ref{proposition: expansion 2}, as such an upper bound guarantees that at least for some parameter values, the clusters will typically be small.
 
\begin{lemma}\label{lemma: upper bound of expansion}
	Let \( \beta,\kappa \geq 0,\) let \( k \geq 2, \) let \( \gamma \) be a path, and let \( p_0 \in C_0(B_N)^+.\) 
	Then there is \( a_m \geq 1 \) such that for every \( P \subseteq C_2(B_N)^+ \) and \( \theta' \in \Omega_1(B_N,\mathbb{Z}_k) \)  such that \( (P,\theta') \) is connected, we have
	\begin{align*}
		\bigl| \phi_\gamma(P, \theta') \bigr|
		&\leq
		(b_{0,k}/b_{j,k})^{|\{ e \in \gamma \colon e \sim (P, \theta') \}|} \prod_{p\in (\support  \theta')^+}e^{2\beta (\Re \rho(\theta'(\partial p)) -1)}
		\\&\qquad\qquad\cdot  
		 \prod_{p \in (P,\theta')}
		 \biggl(
		\int 
		 \bigl| e^{2\beta \Re e^{i \hat \theta'[\partial p]}(e^{ i \theta[\partial p])}-1)}-1\bigr|^{a_m} \prod_{e \in C_1(B_N)^+}\frac{e^{2\kappa \Re  \rho(\theta(e))  }}{
		b_{0,k}}
		\, d\mu(\theta)
		 \biggr)^{\mathrlap{1/a_m}}.
	\end{align*}
	In particular, if we let \( |(P,\theta')| \coloneqq |P \cup (\support \theta')^+| \), then for any \( \kappa > 0 \) we have
	\begin{equation*}
		\sup_{(P,\theta')} \bigl| \phi_\gamma(P, \theta') \bigr|^{1/|(P, \theta')|} \leq C_{\beta,\kappa} \searrow 0 \quad \text{as } \beta \to 0.
	\end{equation*}
	and for any \( \beta > 0 \) we have
	\begin{equation*}
		\sup_{(P,\theta')} \bigl| \phi_\gamma(P, \theta') \bigr|^{1/|(P, \theta')|} \leq C_{\beta,\kappa}' \searrow 0 \quad \text{as } \kappa \to \infty.
	\end{equation*}
\end{lemma}

\begin{proof}
	Let \( P \subseteq C_2(B_N)^+ \) and \( \theta' \in \Omega_1(B_N,\mathbb{Z}_k) \). Then 
	\begin{align*}
		\bigl| \phi(P, \theta') \bigr| 
		&\leq
		(b_{0,k}/b_{j,k})^{|\{ e \in \gamma \colon e \sim (P, \theta') \}|} \prod_{p\in (\support  \theta')^+}e^{2\beta (\Re \rho(\theta'(\partial p)) -1)}
		\\&\qquad\qquad\cdot 
		\prod_{p \in  (P, \theta')}
		 \biggl(
		\int 
		 \bigl| e^{2\beta \Re e^{i \hat \theta'[\partial p]}(e^{ i \theta[\partial p])}-1)}-1\bigr|^{a_m} \prod_{e \in C_1(B_N)^+}\frac{e^{2\kappa \Re  \rho(\theta(e))  }}{
		b_{0,k}}
		\, d\mu(\theta)
		 \biggr)^{|(P,\theta')/a_m}.
	\end{align*}
	Next, note that there is an absolute constant \( a_m \) (that depends only on the dimension of the lattice \( \mathbb{Z}^m \)) such that \( C_2(B_N)^+ \) can be partitioned into disjoint sets \( P_1, \dots, P_{a_m} \) such that for each set, no two plaquettes have a common edge in their boundary.
		Hence, by H\"olders inequality, applied with the probability measure \( \prod_{e \in C_1(B_N)^+}\frac{e^{2\kappa \Re  \rho(\theta(e))  }}{
		b_{0,k}} \, d\mu(\theta) \), we have 
	\begin{align*}
		&
		\int 
		 \prod_{p \in (P, \theta')} \bigl| e^{2\beta \Re \rho(\theta'(\partial p)) (\rho(\theta(\partial p)) -1)}-1\bigr| \prod_{e \in C_1(B_N)^+}\frac{e^{2\kappa \Re  \rho(\theta(e))  }}{
		b_{0,k}} \, d\mu(\theta) 
		 \\&\qquad\leq
		 \prod_{i=1}^{a_m}
		 \biggl(
		\int 
		 \prod_{p \in P_i \cap (P, \theta')} \bigl| e^{2\beta \Re e^{i \hat \theta'[\partial p]}(e^{ i \theta[\partial p])}-1)}-1\bigr|^{a_m} \prod_{e \in C_1(B_N)^+}\frac{e^{2\kappa \Re  \rho(\theta(e))  }}{
		b_{0,k}}
		\, d\mu(\theta)
		 \biggr)^{1/a_m}
		 \\&\qquad=
		 \prod_{i=1}^{a_m}
		 \prod_{p \in P_i \cap (P, \theta')} 
		 \biggl(
		\int 
		 \bigl| e^{2\beta \Re e^{i \hat \theta'[\partial p]}(e^{ i \theta[\partial p])}-1)}-1\bigr|^{a_m} \prod_{e \in C_1(B_N)^+}\frac{e^{2\kappa \Re  \rho(\theta(e))  }}{
		b_{0,k}}
		\, d\mu(\theta)
		 \biggr)^{1/a_m}
		 \\&\qquad=
		 \prod_{p \in  (P, \theta')}
		 \biggl(
		\int 
		 \bigl| e^{2\beta \Re e^{i \hat \theta'[\partial p]}(e^{ i \theta[\partial p])}-1)}-1\bigr|^{a_m} \prod_{e \in C_1(B_N)^+}\frac{e^{2\kappa \Re  \rho(\theta(e))  }}{
		b_{0,k}}
		\, d\mu(\theta)
		 \biggr)^{1/a_m}.
	\end{align*}
	This concludes the proof.  
\end{proof}
 
\clearpage

\section{Proof of main results}\label{section: proof of main result}

In this section, we give proofs of our main results, Theorem~\ref{theorem: main result charge nondiv} and Theorem~\ref{theorem: main result charge div}, of which Theorem~\ref{theorem: U1 Wilson line ratios MF} follows as a special case. The proofs of these theorems will be divided into several propositions throughout this section, which are finally summarized in Section~\ref{sec: summary}.

\subsection{Upper and lower bounds of Wilson~lines and Wilson~loops}

In this section, we obtain upper and lower bounds of Wilson~lines and loops, which together give a basic understanding of how the abelian lattice Higgs model behaves for different charges. The main tools in this section are the current expansion from Section~\ref{sec: current expansion} and the Griffiths-Hurst-Sherman inequality. 

The first result of this section is Lemma~\ref{lemma: positivity} below, which shows that for any path \( \gamma\), \( W_{j\gamma} \) has positive expected value if any only if the charge \( k \) satisfies \( k \mid j. \)

\begin{lemma}\label{lemma: positivity}
	Let \( \beta,\kappa >0, \) let  \( \gamma \) be an open path,  let \( j \geq 1 \), and let \( k \geq 0. \)
	Then \begin{equation*}
		\mathbb{E}_{N,\beta,\kappa,k}[W_{j\gamma}] \geq 0.
	\end{equation*}
	Moreover,
	\begin{equation*}
		\mathbb{E}_{N,\beta,\kappa,k}[W_{j\gamma}] > 0
	\end{equation*}
	if and only if \( k \geq 1 \) and \( k  \mid j. \)
\end{lemma}

\begin{proof}
	The first claim of the lemma is a direct consequence of the Griffiths-Hurst-Sherman inequality, or alternatively, of Proposition~\ref{proposition: current for Higgs with charge}. For the second claim of the lemma, note that by~Proposition~\ref{proposition: current for Higgs with charge}, we have \( \mathbb{E}_{N,\beta,\kappa,k}[W_{j\gamma}] > 0 \) if any only if \( \mathcal{C}_{j\gamma ,k}\) is non-empty, i.e., if there is~\( \mathbf{n}\in \mathcal{C} \) such that
	\begin{equation}\label{eq: C gamma k condition}
		 (\mathbf{n}[\hat \partial e]-\mathbf{n}[-\hat \partial e]) + k(\mathbf{n}[e]-\mathbf{n}[-e]) +  j\gamma[e]= 0,  \quad \forall e\in C_1(B_N).  
	\end{equation}

	Now assume that \( \mathbf{n} \in \mathcal{C}\) satisfies~\eqref{eq: C gamma k condition}.
	 For \( e \in C_1(B_N), \) define \( \sigma(e) \coloneqq \mathbf{n}[e]-\mathbf{n}[-e]\), and for \( p \in C_2(B_N), \) define \( \omega (p) \coloneqq \mathbf{n}[p] - \mathbf{n}[-p]. \) Then \(  \sigma \in \Omega_1(B_N,\mathbb{Z}) \) and \(  \omega \in \Omega_2(B_N,\mathbb{Z}) \). Since~\( \mathbf{n} \)~satisfies~\eqref{eq: C gamma k condition}, with this notation, we have 
	 \begin{align*}
	 	\delta \omega  + k \sigma    +  j\gamma= 0 	 ,
	 \end{align*}
	 implying in particular that 
	 \begin{align*}
	 	k {\delta \sigma} + j \partial \gamma= 0.
	 \end{align*}
	 Since \( \gamma \) is an open path, \( j\partial \gamma \) is a non-trivial \( 0\)-form, taking values in \( \{ -j,0,j\},\) and hence we must have \( k \geq 1 \) and  \( k \mid j. \)

	 We now show that \( k \geq 1 \) and  \( k \mid j. \), then  \( \mathcal{C}_{j\gamma,k} \) is non-empty. To this end, assume that \( k \geq 1 \) and \( k \mid j. \) Define \( \mathbf{n} \in \mathcal{C}\) by letting \( \mathbf{n}[p] = 0 \) for all \( p \in C_2(B_N), \) and for \( e \in C_1(B_N), \) letting
	 \begin{equation*}
	 	\mathbf{n}[e] = \begin{cases}
	 		 (j/k) \gamma[-e] &\text{if }\gamma[-e] \geq 0 \cr 
	 		0 &\text{otherwise.}
	 	\end{cases}
	 \end{equation*}
	 Then, by definition, for any \(  e\in C_1(B_N) \) we have \( (\mathbf{n}[\hat \partial e]-\mathbf{n}[-\hat \partial e]) = 0 \) and 
	 \begin{equation*}
	 	k(\mathbf{n}[e]-\mathbf{n}[-e]) = -j\gamma[e],
	 \end{equation*}
	 and hence
	 \begin{equation*} 
		 (\mathbf{n}[\hat \partial e]-\mathbf{n}[-\hat \partial e]) + k(\mathbf{n}[e]-\mathbf{n}[-e]) +  j\gamma[e]
		 =  
		 0,
	\end{equation*}
	implying in particular that \( \mathbf{n} \in \mathcal{C}_{j\gamma,k}. \) This completes the proof.
\end{proof}

The following result, Proposition~\ref{proposition: always perimeter law}, shows that if \( k \mid j, \) then Wilson~loop observables always obey a perimeter law.

\begin{proposition}\label{proposition: always perimeter law}
	Let \( \beta,\kappa >0, \) let \( \gamma \) be a path, and let \( j ,k \geq 1 \) be such that \( k \mid j. \) Then
	\begin{align*}
		\langle W_{j\gamma} \rangle_{\beta,\kappa,k } > \langle W_{e_0} \rangle_{\beta,\kappa,k }^{j|\gamma|} >0.
	\end{align*}
	Hence, if \( \gamma \) is a loop, the Wilson~loop observable \( W_{j \gamma} \) has perimeter law for all \( \beta,\kappa >0. \)
\end{proposition}

\begin{proof}
	By the Griffiths-Hurst-Sherman inequality, we have 
	\begin{equation*}
		\mathbb{E}_{N,\beta,\kappa,k}[W_{j\gamma}] \geq \prod_{e \in \gamma} \mathbb{E}_{N,\beta,\kappa,k}[W_{je}]. 
	\end{equation*}
	Letting \( N \to \infty \) and using Lemma~\ref{lemma: positivity}, the desired conclusion immediately follows.
\end{proof}

Motivated by Proposition~\ref{proposition: always perimeter law}, we now turn to the case \( k \not \mid j, \) and show that in this case, the abelian lattice Higgs model has a phase transition between a region with area law and a region with perimeter law. To this end, we first prove that for any \(j \) and \( k \), there is always a region in the parameter space where the model has a perimeter law.

\begin{proposition}\label{proposition: always a region with perimeter law}
	Let \( \beta_c \) be so that pure lattice gauge theory with \( \beta> \beta_c \) has perimeter law.
	Let \( \beta > \beta_c  \) and \( \kappa > 0 \), let \( \gamma \) be a loop, and let \( j \geq 1 \) and \( k \geq 0. \) Then there is \( C,c > 0 \) such that \( \langle W_{j \gamma} \rangle_{\beta,\kappa,k} \geq Ce^{-c|\gamma|} \) for every \( n. \)
\end{proposition}

\begin{proof}
	By the Griffiths-Hurst-Sherman inequality, we have
	\begin{equation*}
		\mathbb{E}_{N,\beta,\kappa,k}[W_{j,\gamma}] \geq \mathbb{E}_{N,\beta,0,k}[W_{j,\gamma}] = \mathbb{E}_{N,\beta}[W_{j,\gamma}] \geq \mathbb{E}_{N,\beta}[W_{\gamma}]^j.
	\end{equation*}
	Since pure lattice gauge theory has perimeter law for all sufficiently large \( \beta \) (see, e.g., \citep{fv2025b}), the desired conclusion follows.
\end{proof}

We next show that if \( k \not \mid j,\) then there exists a region in the parameter space where the abelian lattice Higgs model exhibits an area law.

\begin{proposition}\label{proposition: area law}
	Let \( j \geq 1 \) and \( k \geq 0 \), let \( \gamma \) be a rectangular loop, and let \( j,k \geq 1 \) be such that \( k \not \mid j.\) Then there is \( \beta_c '> 0 \), \( C = C(\beta,\kappa)>0, \) and \( c=c(\beta,\kappa) > 0 \) such that if \( \beta \in (0,\beta_c') \) and \( \kappa \geq 0 \) then   \( \mathbb{E}_{N,\beta,\kappa,k}[W_{j\gamma}] \leq Ce^{-c \area(\gamma)}.\)
\end{proposition}

\begin{proof}

	Recall first that  \( \mathbf{n} \in \mathcal{C}_{j\gamma,k}. \) if and only if 
	 \begin{equation*} 
		 (\mathbf{n}[\hat \partial e]-\mathbf{n}[-\hat \partial e]) + k(\mathbf{n}[e]-\mathbf{n}[-e]) +  j\gamma[e]
		 =  
		 0,
	\end{equation*}
	
	Let \( \mathcal{C}_{j\gamma,k}^0 \) be the set of \( \mathbf{n} \in  \mathcal{C}_{j\gamma,k} \) that are minimal in the sense that there is no \( \mathbf{m} \leq \mathbf{n} \) such that \( \mathbf{m} \in \mathcal{C}_{j\gamma,k}^0. \) This allows us to write
	\begin{align*}
		\sum_{\mathbf{n} \in \mathcal{C}_{j\gamma,k}} w(\mathbf{n}) 
		\leq 
		\sum_{\mathbf{m} \in \mathcal{C}_{j\gamma,k}^0}  \sum_{\substack{\mathbf{n} \in \mathcal{C}_{j\gamma,k} \mathrlap{\colon} \\ \mathbf{m}\leq \mathbf{n}}} w(\mathbf{n}) 
		= 
		\sum_{\mathbf{m} \in \mathcal{C}_{j\gamma,k}^0}  \sum_{\substack{\mathbf{m}' \in \mathcal{C}_{0,k}  }} w(\mathbf{m}+\mathbf{m}') 
		\leq
		\sum_{\mathbf{m} \in \mathcal{C}_{j\gamma,k}^0}  w(\mathbf{m})\sum_{\substack{\mathbf{m}' \in \mathcal{C}_{0,k}  }} w(\mathbf{m}'),
	\end{align*}
	which implies, in particular, that
	\begin{align*}
		\mathbb{E}_{N,\beta,\kappa,k}[W_{j\gamma}] \leq \sum_{\substack{\mathbf{m}' \in \mathcal{C}_{0,k}  }} w(\mathbf{m}').
	\end{align*}
	Now note that by definition, the support of any \( \mathbf{n} \in \mathcal{C}_{j\gamma,k}^0\) is connected, and moreover, this is still true if we restrict the support to \( C_2(B_N).\) Moreover, since \( k \not \mid j, \) the support of any \( \mathbf{n} \in \mathcal{C}_{j\gamma,k}^0 \) must contain at least \( \area(\gamma) \) plaquettes. The desired conclusion now follows from standard enumeration arguments. 
\end{proof}
	
\subsection{The Marcu--Fredenhagen ratio}

In this section, we state and prove our main results about the Marcu--Fredenhagen ratio for the case \( k \mid j \geq 1. \)  Note that if \( k \not \mid j, \) then, by~Lemma~\ref{lemma: positivity}, the Marcu--Fredenhagen ratio is identically zero.

 \subsubsection{The free phase}
 In this section, we consider the free phase, corresponding roughly to the part of the phase diagram where \( \beta \) is large and \( \kappa \) is small.
 
\begin{proposition}
\label{proposition: free}
	Let \( j , k \geq 1 \) and assume that \( k \mid j. \) Let \( \beta_0 \) be the phase transition between perimeter law and area law in the pure model. Then for every \( \beta>\beta_0 \) and \( \kappa \) sufficiently small compared to \( \beta, \) we have \( \lim_{n \to \infty} r_{n,k,j} = 0. \)
\end{proposition}

\begin{remark}
	The proof of Proposition~\ref{proposition: free} below closely follows the proof of for the case \( j = k = 1 \) in~\citep{fm1983}, and can also be found in~\citep[Section 1.2]{b1990}. 
\end{remark}

\begin{proof}[Proof of Proposition~\ref{proposition: free}]
	We first give an upper bound on the numerator of \( r_{n,k,j}. \) To this end, note first that by the Griffiths-Hurst-Sherman inequality, we have
	\begin{align*}
		& \mathbb{E}_{N,\beta,\kappa,k} [W_{j\gamma_n}]\mathbb{E}_{N,\beta,\kappa,k} [W_{j\gamma_n'}] 
		=
		\mathbb{E}_{N,\beta,\kappa,k} [W_{j\gamma_n}]^2
		\leq  
		\mathbb{E}_{N,\infty,\kappa,k} [W_{j\gamma_n}]^2
		=
		\mathbb{E}_{N,\infty,\kappa,1} [W_{(j/k)\gamma_n}]^2. 
	\end{align*}
	Note that when \( \beta= \infty, \) then \( \sigma \sim \mu_{N,\infty,\kappa,1} \) has the same distribution as \( d\eta, \) where \( \eta \) is the XY-model on \( \Omega_0(B_N,U(1)). \) Consequently, ff \( \kappa \) is sufficiently small, then the probability on the right-hand side of the previous equation decays like \( e^{-c_\kappa |y_n-x_n|}, \) where \( c_\kappa \) can be made arbitrarily large by choosing \( \kappa \) small.

	We now give a lower bound on the denominator of \( r_{n,k,j}. \) To this end, note that by the Griffiths-Hurst-Sherman inequality, we have
	\begin{align*}
		\mathbb{E}_{N,\beta,\kappa,k} [W_{j(\gamma_n+\gamma_n')}] \geq 
		\mathbb{E}_{N,\beta,0,k} [W_{j(\gamma_n+\gamma_n')}] \geq 
		\mathbb{E}_{N,\beta,0,k} [W_{\gamma_n+\gamma_n'}]^j
	\end{align*}
	Note that when \( \kappa = 0, \) we get the pure \( U(1) \) lattice gauge theory. Consequently,  if \( \beta \) is sufficiently large, then  \( \mathbb{E}[W_{\gamma_n+\gamma_n'}] \) decays like \( e^{-c'_\beta |\gamma|} , \) where \( c'_\beta \) can be made arbitrarily small by choosing \( \beta \) large (see, e.g., \citep{fs1982}). 
	
	Combining the upper and lower bounds given above, the desired conclusion immediately follows.
\end{proof}

\subsubsection{The confinement regime}

In this section, we consider the confinement regime, which roughly corresponds to the part of the phase diagram where \( \beta\) is small.
The primary tool we employ in this regime is the current expansion, as described in Section~\ref{sec: current expansion}, which enables us to map the abelian lattice Higgs model to a discrete model. We then use disagreement percolation (see, e.g., \citep{b1993}) to lower bound the Marcu--Fredenhagen ratio. An alternative route for proving this result would be to instead first use either the current expansion or the expansion described in~Section~\ref{section: new version of expansion}, and then proceed by using a cluster expansion of \( \log Z[\gamma] \) (see, e.g., \citep[Chapter 5]{fv2007}). However, the proof presented here, using disagreement percolation, requires the expansion to assign positive weight to clusters and thus cannot be used in conjunction with the signed expansion of~Section~\ref{section: new version of expansion}. 

\begin{proposition}[MF ratio in the confinement regime]\label{proposition: confinement}
	Let \( \beta,\kappa > 0, \) and let \( j , k \geq 1 \) be such that \( k \mid j.\) Then, if \( \bigl(1+16(d-1)\bigr)^{2} ( e^{j\beta /k  }-1) e^{4\kappa} \)
	is sufficiently small, we have \( \liminf_{n \to \infty} r_{n,k,j} > 0. \) 
\end{proposition}

 We illustrate the region where Proposition~\ref{proposition: confinement} applies in Figure~\ref{fig: currentweight}.

\begin{figure}[h]
	\centering
	\includegraphics[width=.5\textwidth]{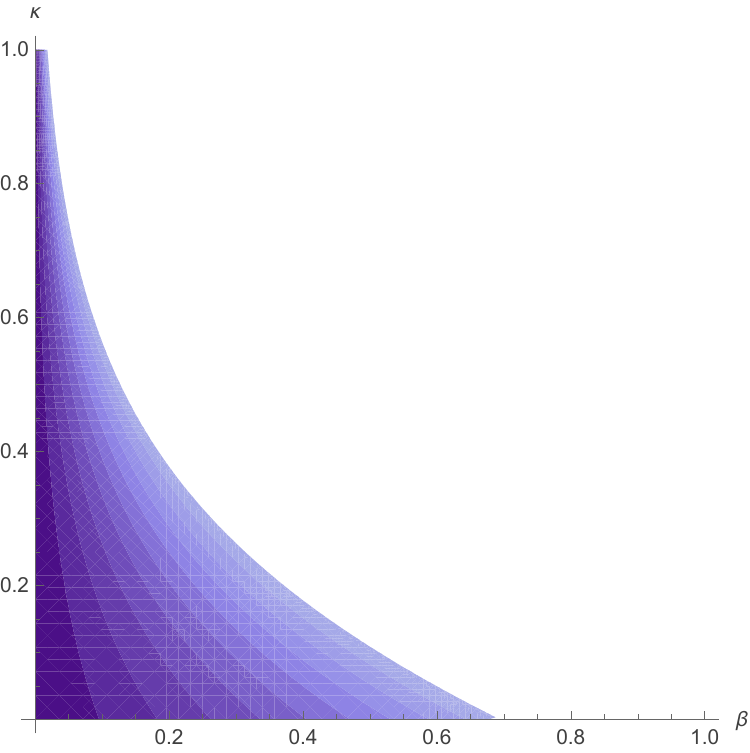}
	\caption{The figure above shows the level sets of the function on the right hand side of the function \( (e^{\beta}-1)e^{4\kappa} \), which equivalently is the level sets of the upper bound in Proposition~\ref{proposition: confinement}.} \label{fig: currentweight}
\end{figure}

In the proof of Proposition~\ref{proposition: confinement}, we will use the current expansion described in Section~\ref{sec: current expansion}. In addition, we will use the following notation.

When \( \mathbf{n} \in \mathcal{C} \), we let \( \mathcal{G}(\mathbf{n}) \) be the graph with vertex set \( \support \mathbf{n} \), an edge between two plaquettes \( p_1,p_2 \in \support \mathbf{n} \cap C_2(B_N) \) if \( \support \partial p_1 \cap \support \partial p_2 \neq \emptyset, \) and an edge between \( p \in \support \mathbf{n} \cap C_2(B_N) \)  and \( e \in \support \mathbf{n} \cap C_1(B_N) \) if \( e \in \pm \support \partial p. \)

Whenever \( \mathbf{m},\mathbf{n} \in \mathcal{C}, \) we write \( \mathbf{m} \leq \mathbf{n} \) if the following holds.
\begin{enumerate}
	\item \( \mathbf{n}[c] = \mathbf{m}[c] \) for all \( c \in \support \mathbf{m}, \) and
	\item \( \support \mathbf{m} \) is a connected component of \( \mathcal{G}(\mathbf{n}). \)
\end{enumerate}

We now proceed to the proof of Proposition~\ref{proposition: confinement}.
\begin{proof}[Proof of Proposition~\ref{proposition: confinement}]
	First, we note that for any path \( \gamma \), by~Proposition~\ref{proposition: current for Higgs with charge}, we have
	\begin{align*}
		\mathbb{E}_{N,\beta,\kappa,k}[W_{j\gamma}] = \frac{\sum_{\mathbf{n}\in\mathcal{C}_{j\gamma,k}} w_{\beta,\kappa}(\mathbf{n})}{\sum_{\mathbf{n}\in\mathcal{C}_{0,k}} w_{\beta,\kappa}(\mathbf{n})},
	\end{align*}
	where we recall from Proposition~\ref{proposition: current for Higgs with charge} that
	\begin{equation*}
		\mathcal{C}_{j\gamma,k} \coloneqq \bigl\{ \mathbf{n} \in \mathcal{C} \colon (\mathbf{n}[\hat \partial e]-\mathbf{n}[-\hat \partial e]) + k(\mathbf{n}[e]-\mathbf{n}[-e]) +  j\gamma[e]= 0,  \; \forall e\in C_1(B_N) \bigr\}.
	\end{equation*}
	We will prove the statement by defining a bijective map \( \varphi\) between a set \( \tilde {\mathcal{C}}_{j\gamma_n,j\gamma_n'} \subseteq \mathcal{C}_{j\gamma_n,k}\times \mathcal{C}_{j\gamma'_n,k} \) and a set \( \tilde {\mathcal{C}}_{j(\gamma_n+\gamma_n'),0} \subseteq  \mathcal{C}_{j(\gamma_n+\gamma'_n),k}\times \mathcal{C}_{0,k} \) which preserves \( w_{\beta,\kappa} \), and then show that \( \tilde {\mathcal{C}}_{j(\gamma_n+\gamma_n'),0} \) have positive density \( \alpha \) with respect to  \( \mathbf{P}^{j(\gamma_n+\gamma_n'),k} \times  \mathbf{P}^{0,k}\).
	Since this implies that 
	\begin{align*}
		&\frac{\mathbb{E}_{N,\beta,\kappa,k}[W_{j\gamma_n}]\mathbb{E}_{N,\beta,\kappa,k}[W_{j\gamma_n'}]}{\mathbb{E}_{N,\beta,\kappa,k}[W_{j(\gamma_n+\gamma_n')}]} =
		\frac{w_{\beta,\kappa}(\mathcal{C}_{j\gamma_n,k}) w_{\beta,\kappa}(\mathcal{C}_{j\gamma_n',k})}{w_{\beta,\kappa}(\mathcal{C}_{j(\gamma_n+\gamma_n'),k} )w_{\beta,\kappa}(\mathcal{C}_{0,k}) }  
		\geq
		\frac{w_{\beta,\kappa}(\tilde{\mathcal{C}}_{j\gamma_n,j\gamma'_n})}{w_{\beta,\kappa}(\mathcal{C}_{j(\gamma_n+\gamma_n'),k} )w_{\beta,\kappa}(\mathcal{C}_{0,k}) } 
		\\&\qquad=
		\frac{w_{\beta,\kappa}(\tilde{\mathcal{C}}_{\gamma+\gamma',0})}{w_{\beta,\kappa}(\mathcal{C}_{\gamma+\gamma',k} )w_{\beta,\kappa}(\mathcal{C}_{0,k}) } 
		=
		\frac{\alpha w_{\beta,\kappa}(\mathcal{C}_{j(\gamma+\gamma'),k} )w_{\beta,\kappa}(\mathcal{C}_{0,k})}{w_{\beta,\kappa}(\mathcal{C}_{j(\gamma_n+\gamma_n'),k)} )w_{\beta,\kappa}(\mathcal{C}_{0,k}) } =\alpha >0,
	\end{align*}
	this implies the desired conclusion.
	
	We now define \( \varphi, \) \( \tilde{\mathcal{C}}_{j\gamma_n,j\gamma_n'} \) and \( \tilde {\mathcal{C}}_{j(\gamma_n+\gamma_n'),0}.\) 
	Given \( \mathbf{n}_1 \in \mathcal{C}_{j\gamma_n,k} \), \( \mathbf{n}_2 \in \mathcal{C}_{j\gamma_n',k} \), and a path \( \tilde \gamma, \) let  \( A_{\tilde \gamma}(\mathbf{n}_1,\mathbf{n}_2) \) be the set of all plaquettes that are in some connected component of \( \mathcal{G}( \mathbf{n}_1 +\mathbf{n}_2) \) that is adjacent to at least one of the edges in \( \tilde \gamma\). Further, let
	\[
		\Lambda_{\tilde \gamma}(\mathbf{n}_1,\mathbf{n}_2) \coloneqq A_{\tilde \gamma}(\mathbf{n}_1,\mathbf{n}_2) \cup \{ e \subseteq C_1(B_N) \colon \pm \support \hat \partial e \cap A_{\tilde \gamma}(\mathbf{n}_1,\mathbf{n}_2) \neq \emptyset \},
	\]  
	\begin{equation*}
		\tilde {\mathcal{C}}_{j\gamma_n,j\gamma_n'}   \coloneqq 
		\bigl\{  
		(\mathbf{n}_1,\mathbf{n}_2 ) \in \mathcal{C}_{j\gamma_n,k}\times \mathcal{C}_{j\gamma'_n,k} \colon 
		 A_{j\gamma_n}(\mathbf{n}_1,\mathbf{n}_2 ) \cap A_{j\gamma_n'}(\mathbf{n}_1,\mathbf{n}_2 ) = \emptyset
		\bigr\}
	\end{equation*}
	and, similarly, let
	\begin{equation*}
		\tilde {\mathcal{C}}_{j(\gamma_n+\gamma_n'),0}  \coloneqq 
		\bigl\{  
		(\mathbf{n}_1,\mathbf{n}_2 ) \in \mathcal{C}_{j(\gamma_n+\gamma_n'),k}\times \mathcal{C}_{0,k} \colon 
		 A_{j\gamma_n}(\mathbf{n}_1,\mathbf{n}_2 ) \cap A_{j\gamma_n'}(\mathbf{n}_1,\mathbf{n}_2 ) = \emptyset
		\bigr\}.
	\end{equation*} 
	Then
	\[
		 \varphi \colon (\mathbf{n}_1 , \mathbf{n}_2  ) \mapsto (\mathbf{n}_1 \mid \Lambda_{j\gamma_n} +  \mathbf{n}_2|_{\Lambda_{j\gamma_n}^c} ,\mathbf{n}_2 \mid \Lambda_{j\gamma_n}+  \mathbf{n}_1|_{\Lambda_{j\gamma_n}^c})
	\]
	is a bijective map from \(  \tilde {\mathcal{C}}_{j\gamma_n,j\gamma_n'}  \) to \(  \tilde {\mathcal{C}}_{j(\gamma_n+\gamma_n'),0} . \)  
	Moreover, since for every \( c \in C_1(B_N) \cup C_2(B_N) \) we have
	\[
		\bigl\{ \mathbf{n}_1(c),\mathbf{n}_2(c) \bigr\} = \bigl\{ \varphi(\mathbf{n}_1,\mathbf{n_2})_1(c),\varphi(\mathbf{n}_1,\mathbf{n_2})_2(c) \bigr\},
	\]
	it follows that
	\[
		w_{\beta,\kappa}\bigl(\varphi(\mathbf{n}_1,\mathbf{n}_2)_1\bigr)+w_{\beta,\kappa}\bigl(\varphi(\mathbf{n}_1,\mathbf{n}_2)_2\bigr) = w_{\beta,\kappa}(\mathbf{n}_1) + w_{\beta,\kappa}(\mathbf{n}_2).
	\]

	It now only remains to show that if \( \beta \) and \( \kappa \) are both sufficiently small, then 	\[
		\alpha = (\mathbf{P}^{j(\gamma_n+\gamma_n'),k} \times  \mathbf{P}^{0,k} )\bigl(\tilde {\mathcal{C}}_{j(\gamma_n+\gamma_n'),0} \bigr)  > 0 .
	\] 
	To this end, note first that by definition, we have
	\begin{align*}
		&(\mathbf{P}^{j(\gamma_n+\gamma_n'),k} \times  \mathbf{P}^{0,k})(\tilde {\mathcal{C}}_{j(\gamma_n+\gamma_n'),0}) 
		=
		(\mathbf{P}^{j(\gamma_n+\gamma_n'),k} \times  \mathbf{P}^{0,k} )\bigl( A_{j\gamma_n}(\mathbf{n}_1,\mathbf{n}_2 ) \cap A_{j\gamma_n'}(\mathbf{n}_1,\mathbf{n}_2 ) = \emptyset \bigr) 
	\end{align*} 
	On the event \(   A_{j\gamma_n}(\mathbf{n}_1,\mathbf{n}_2 ) \cap A_{j\gamma_n'}(\mathbf{n}_1,\mathbf{n}_2 ) = \emptyset, \) there must exist a connected component in \( \mathcal{G}( \mathbf{n}_1 + \mathbf{n}_2) \) that is adjacent to both \( \gamma_{n} \) and \( \gamma_n'. \) To upper bound the probability of this event, for an edge \( e \in \gamma_n \), let \( \mathcal{C}_e(\mathbf{n}_1,\mathbf{n}_2) \) be the connected component in \( \mathcal{G}( \mathbf{n}_1 + \mathbf{n}_2) \) that is adjacent to~\( e. \) Then 
	\begin{align*}
		&(\mathbf{P}^{j(\gamma_n+\gamma_n'),k} \times  \mathbf{P}^{0,k} ) \bigl( A_{j\gamma_n}(\mathbf{n}_1,\mathbf{n}_2 ) \cap A_{j\gamma_n'}(\mathbf{n}_1,\mathbf{n}_2 ) \neq \emptyset  \bigr)
		\\&\qquad \leq 
		\sum_{e \in \gamma_n} 
		(\mathbf{P}^{j(\gamma_n+\gamma_n'),k} \times  \mathbf{P}^{0,k}) \bigl( |\mathcal{C}_e| \geq \dist(e,\gamma_n') \bigr).
	\end{align*}

	Now assume that \( \mathbf{m}  \in \mathcal{C}_{0,k} \) is given. Then, for any \( \mathbf{n} \in \mathcal{C}_{0,k} \) such that \( \mathbf{m}\leq \mathbf{n}\) we have \( \mathbf{n}-\mathbf{m} \in \mathcal{C}_{0,k}.\) Also, since \( \mathbf{m} \) and \( \mathbf{n}- \mathbf{m} \) have disjoint supports, it follows that \( w_{\beta,\kappa}(\mathbf{n}) = w_{\beta,\kappa}(\mathbf{m}) w_{\beta,\kappa}(\mathbf{n-m}) )\).  Hence
	\begin{align*}
		\mathbf{P}_{\beta,\kappa}^{0,k}(\mathbf{m} \leq \mathbf{n}) = \frac{\sum_{\mathbf{n} \in \mathcal{C}_{0,k} \colon \mathbf{m} \leq \mathbf{n}} w_{\beta,\kappa}(\mathbf{n})}{\sum_{\mathbf{n} \in \mathcal{C}_{0,k} } w_{\beta,\kappa}(\mathbf{n})} 
		\leq 
		\frac{\sum_{\mathbf{n} \in \mathcal{C}_{0,k} \colon \mathbf{m} 
		\leq
		\mathbf{n}} w_{\beta,\kappa}(\mathbf{n})}{\sum_{\mathbf{n} \in \mathcal{C}_{0,k} \colon \mathbf{m} 
		= 
		\mathbf{n} } w_{\beta,\kappa}(\mathbf{n}-\mathbf{m})} 
		\leq 
		w_{\beta,\kappa}(\mathbf{m}) .
	\end{align*}
	Next, assume \( \mathbf{m} \in \mathcal{C} \) is such that \( \mathbf{m} \leq \mathbf{n} \in  \mathcal{C}_{j(\gamma+\gamma'),k} \). Let \( \gamma'' \) be the restriction of \( \gamma+\gamma' \) set of edges in \( \gamma+\gamma' \) that are adjacent to \( \mathbf{e}. \) 
	 Let \( \mathbf{m}_0 \) be defined by 
	\begin{equation*}
		\mathbf{m}_0[c] = 
		\begin{cases}
			j\gamma[e]/k &\text{if } c \in \gamma'' \cr 
			0 &\text{}else.
		\end{cases}
	\end{equation*} 
	Then \( \mathbf{n}-\mathbf{m} + \mathbf{m}_0 \in \mathcal{C}_{j (\gamma+\gamma'),k} . \)
	Moreover, we note that since the support of \( \mathbf{n}-\mathbf{m} \) is disjoint form the support of both \( \mathbf{m} \) and \( \mathbf{m}_0, \) we have both
	\[
	 w_{\beta,\kappa}(\mathbf{n})=  w_{\beta,\kappa}(\mathbf{n}-\mathbf{m})w_{\beta,\kappa}(\mathbf{m})\]
	 and
	 \[
	 w_{\beta,\kappa}(\mathbf{m}-\mathbf{m}+\mathbf{m}_0)=
	 w_{\beta,\kappa}(\mathbf{m}-\mathbf{m})
	 w_{\beta,\kappa}(\mathbf{m}_0).
	 \] 
	 Combining these observations, we obtain
	\begin{align*}
		&\mathbf{P}_{\beta,\kappa}^{j(\gamma+\gamma'),k}(\mathbf{m} \leq \mathbf{n}) = 
		\frac{\sum_{\mathbf{n} \in \mathcal{C}_{j(\gamma+\gamma'),k} \colon \mathbf{m} \leq \mathbf{n}} w_{\beta,\kappa}(\mathbf{n})}{\sum_{\mathbf{n} \in \mathcal{C}_{j(\gamma+\gamma'),k} } w_{\beta,\kappa}(\mathbf{n})} 
		\leq 
		\frac{\sum_{\mathbf{n} \in \mathcal{C}_{j(\gamma+\gamma'),k} \colon \mathbf{m} \leq \mathbf{n}} w_{\beta,\kappa}(\mathbf{n})}{\sum_{\mathbf{n} \in \mathcal{C}_{j(\gamma+\gamma'),k} \colon \mathbf{m}\leq \mathbf{n}} w_{\beta,\kappa}(\mathbf{n}-\mathbf{m}+ \mathbf{m}_0)} 
		\\&\qquad = 
		\frac{\sum_{\mathbf{n} \in \mathcal{C}_{j(\gamma+\gamma'),k} \colon \mathbf{m} \leq \mathbf{n}} w_{\beta,\kappa}(\mathbf{n}-\mathbf{m})w_{\beta,\kappa}(\mathbf{m})}{\sum_{\mathbf{n} \in \mathcal{C}_{j(\gamma+\gamma'),k} \colon \mathbf{m}\leq \mathbf{n}} w_{\beta,\kappa}(\mathbf{n}-\mathbf{m} )w_{\beta,\kappa}( \mathbf{m}_0)} 
		=
		w_{\beta,\kappa} (\mathbf{m}) /w_{\beta,\kappa}(\mathbf{m}_0).
	\end{align*} 
	We now claim that there is \( \varepsilon>0 \) such that 
	\[
		w_{\beta,\kappa}(\mathbf{m} )/w_{\beta,\kappa}(\mathbf{m}_0)
		\leq 
		w_{\beta+\varepsilon_n,\kappa}(\mathbf{m}).
		\]	
	To see this, note first that
		\[
		\mathbf{m}_0[p] =0\quad \text{for all } p \in C_2(B_N).
		\]	
		Next, note that for any \( e \in \gamma'' , \) we have
		\begin{align*}
			(\mathbf{m}[\hat \partial e] - \mathbf{m}[-\hat \partial e]) + k(\mathbf{m}[e] -\mathbf{m}[-e]) = k \mathbf{m}''[e].
		\end{align*}
		Also, note that if \( |\gamma_2| \leq \min (2R_n,T_n), \) then \( \sum_{e \in C_1(B_N)} \mathbf{m}[e]  \geq \sum_{e \in C_1(B_N)} \mathbf{m}_0[e] . \)  
		Together, these observations imply that if \(  \support_2 \mathbf{m} \leq \min (2R_n,T_n) \), then
			\[
		w_{\beta,\kappa}(\mathbf{m} )/w_{\beta,\kappa}(\mathbf{m}_0) 
		\leq 
		w_{j\beta/k,\kappa}(\mathbf{m}).
		\]	
		Similarly, if \( |\gamma_2| \geq \min (2R_n,T_n), \) then
			\[
		w_{\beta,\kappa}(\mathbf{m} )/w_{\beta,\kappa}(\mathbf{m}_0) 
		\leq 
		w_{(j/k + o_n(1))\beta,\kappa}(\mathbf{m}).
		\]	 
	Since \( \sum_{i=1}^\infty \kappa^i/i! \leq 1, \) by standard arguments (see, e.g., \citep{f2022b,f2024b}), for any edge \( e \in \gamma_n \) we have
	\begin{align*}
		&(\mathbf{P}^{j(\gamma_n+\gamma_n'),k} \times  \mathbf{P}^{0,k}) \bigl( |\support_2 \mathcal{C}_e| \geq \dist(e,\gamma_n') \bigr) 
		\\&\qquad \leq 
		\sum_{i = \dist(e,\gamma_n')}^\infty (1+2 \cdot 4 \cdot 2(d-1))^{2i} \biggl( \sum_{\ell=1}^\infty \frac{((j/k + o_n(1))\beta)^{\ell}}{\ell!} \biggr)^i \biggl( \sum_{\ell=0}^\infty \frac{\kappa^{\ell}}{\ell!} \biggr)^{4i}
		= 
		\frac{a^{\dist(e,\gamma_n')}}{1-a},
	\end{align*}
	where \( a \coloneqq \bigl(1+16(d-1)\bigr)^{2} ( e^{(j/k + o_n(1))\beta }-1) e^{4\kappa}.\)  
	Summing over all edges \( e \in \gamma_n ,\) we obtain
	\begin{align*}
		&\mathbf{P}^{j(\gamma_n+\gamma_n'),k} \times  \mathbf{P}^{0,k}(\neg \tilde {\mathcal{C}}_{j(\gamma_n+\gamma_n'),0}) 
		\leq 
		\sum_{e \in \gamma_n} 
		(\mathbf{P}^{j(\gamma_n+\gamma_n'),k} \times  \mathbf{P}^{0,k}) \bigl( |\support_2 \mathcal{C}_e| \geq \dist(e,\gamma_n') \bigr) 
		\\&\qquad
		\leq  2\sum_{i=1}^{\infty} \frac{a^i}{1-a} 
		+
		 \frac{T_n a^{R_n}}{1-a}
		 =
		 \frac{2a}{(1-a)^2} 
		+
		 \frac{T_n a^{R_n}}{1-a}.
	\end{align*}
	and hence
	\begin{align*}
		&\mathbf{P}^{j(\gamma_n+\gamma_n'),k} \times  \mathbf{P}^{0,k}( \tilde {\mathcal{C}}_{j(\gamma_n+\gamma_n'),0})  
		 >
		 1-  \frac{2a}{(1-a)^2} 
		-
		 \frac{T_n a^{R_n}}{1-a}.
	\end{align*}
	From this, the desired conclusion immediately follows.
\end{proof}

We next use the expansion of Section~\ref{section: charge k expansion} together with a cluster expansion to prove that the Marcu--Fredenhagen ratio is non-zero when \( \kappa \) is sufficiently large. 
	We illustrate the region where Proposition~\ref{proposition: higgs regime} applies in Figure~\ref{fig: expansionweight}.

\begin{proposition}[MF ratio in the Higgs regime]\label{proposition: higgs regime}
		Let \( \beta,\kappa > 0, \) and let \( j , k \geq 1 \) be such that \( k \mid j.\) Then, if \(  \sup_{|\theta',P|}\bigl| \phi_\gamma(P, \theta') \bigr|^{1/|(P, \theta')|} \) is sufficiently small, we have \( \liminf_{n \to \infty} r_{n,k,j} > 0. \) 
\end{proposition}

\begin{proof}
	Since the proof, with the setup from Section~\ref{section: new version of expansion}, is almost identical to the proof of~\citep[Theorem 1.1]{f2024b}, we only outline the differences here. 
	The polymers in the proof of~\citep[Theorem 1.1]{f2024b} are replaced with sets \( P \subseteq C_2(B_N) \) corresponding to connected graphs \( \mathcal{G}(P) \), and the weight of each cluster is given by \( \phi_\gamma \). , Lemma~\ref{lemma: upper bound of expansion} guarantees that the cluster expansion converges. The rest of the proof then follows the proof of~\citep[Theorem 1.1]{f2024b}, and the conclusion follows by noting that by Lemma~\ref{lemma: upper bound of expansion}, the weight assigned to a cluster of size at least \( T_n = Tn \) goes to zero exponentially fast in \( n. \)
\end{proof}

\begin{figure}[h]
	\centering
	\includegraphics[width=.5\textwidth]{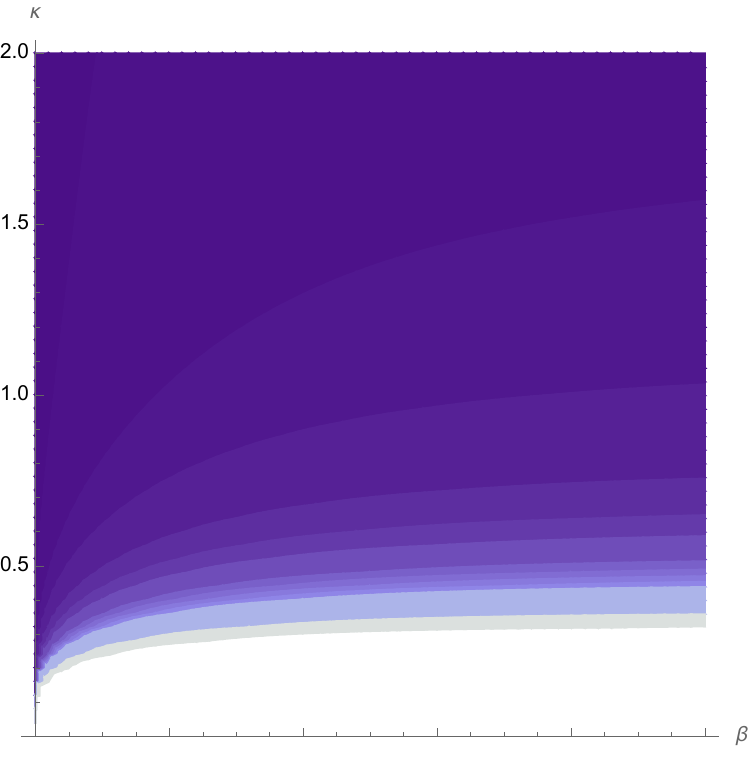}
	\caption{The Figure above shows the level sets of the function on the right hand side of~\eqref{eq: upper Holder bound 1}, which is an upper bound of the quantity that needs to be small for Proposition~\ref{proposition: higgs regime} to be applicable.} \label{fig: expansionweight}
\end{figure}

\begin{remark}
	In essence, having an expansion of \( \log Z[\gamma] \) into a sum over products of well separated clusters whose weight decays exponentially quickly in the size of their support is exactly what is needed to obtain \( \liminf r_n>0. \) In a similar vein, as was already pointed out in~\citep{os1978} and~\citep[Appendix A]{fs1979}, in the same setting one can also easily obtain analyticity of Wilson loop and Wilson line expectation, and also exponential decay of correlations.
\end{remark}

\subsection{Proofs of Theorem~\ref{theorem: main result charge nondiv} and Theorem~\ref{theorem: main result charge div}}\label{sec: summary}
In this section, we provide proofs of the two main results of the paper by providing references to earlier results in the paper that prove their constituent parts.

\begin{proof}[Proof of Theorem~\ref{theorem: main result charge nondiv}]
	The desired conclusion follows by combining the results of this section. In particular, 
	Theorem~\ref{theorem: main result charge nondiv}\ref{theorem: always zero if not divisible} follows by Lemma~\ref{lemma: positivity},  
	Theorem~\ref{theorem: main result charge nondiv}\ref{theorem: area law} is exactly Proposition~\ref{proposition: area law}, and
	Theorem~\ref{theorem: main result charge nondiv}\ref{theorem: always a region with perimeter law nondiv} is exactly Proposition~\ref{proposition: always a region with perimeter law}.
\end{proof}

\begin{proof}[Proof of Theorem~\ref{theorem: main result charge div}]
	The desired conclusion follows directly from combining the results of this section. In~particular, Theorem~\ref{theorem: main result charge div}\ref{theorem: always perimeter law} is exactly Proposition~\ref{proposition: always perimeter law},
	Theorem~\ref{theorem: main result charge div}\ref{theorem: main result div confinement} is exactly Proposition~\ref{proposition: confinement},
	Theorem~\ref{theorem: main result charge div}\ref{theorem: main result div free} is exactly Proposition~\ref{proposition: free}, and, finally, 
	Theorem~\ref{theorem: main result charge div}\ref{theorem: main result div higgs} is exactly Proposition~\ref{proposition: higgs regime}. 
\end{proof}


\begin{thebibliography}{99}



 

\bibitem{a2021} A. Adhikari, Wilson~loop expectations for non-abelian
gauge fields coupled to a Higgs boson at low and high disorder, Commun. Math. Phys. 405, 117 (2024). 

\bibitem{ca2025} A. Adhikari, S. Cao, Correlation decay for finite lattice gauge theories at weak coupling, Ann. Probab. 53(1): 140-174 (2025).  

	
	\bibitem{bbij1984} T. Balaban, D. Brydges, J. Imbrie, A. Jaffe, The mass gap for Higgs models on a unit lattice, Ann. Phys. 158, 281-319 (1984).
	
\bibitem{b1990} J. C. A. Barata, On the phase structure of the compact abelian lattice Higgs model, Commun. Math. Phys. 129, 511--523 (1990).

\bibitem{b1993} J. van den Berg, A uniqueness condition for Gibbs measures, with application to the {2-dimensional} Ising antiferromagnet, Commun. Math. Phys. 152, 161-166 (1993).

	\bibitem{bn1987} Borgs, C., Nill, F., The phase diagram of the abelian lattice Higgs model. A Review of Rigorous Results, Journal of Statistical Physics, Vol. 47, Nos. 5/6, (1987).
 
 
	
 
		


 
	\bibitem{c2020} S. Cao, Wilson loop expectations in lattice gauge theories with finite gauge groups. Comm. Math. Phys., 380, 1439--1505, (2020).

	
	\bibitem{c2018} S. Chatterjee, Yang-Mills for Probabilists, In: Friz, P., Konig, W., Mukherjee, C., Olla, S. (eds) Probability and analysis in interacting physical systems. VAR75 2016. Springer Proceedings in Mathematics \& Statistics, vol 283. Springer, Cham., (2020), 307--340.



	\bibitem{c2025} S. Chatterjee, O. Yakir, Correlation decay for \(U(1)\) lattice Higgs theory: the case of small mass, preprint available at \url{https://arxiv.org/abs/2509.19176} (2025).

    
    
	 
    
    
    
    
    
	\bibitem{ds2024} P. Duncan, B. Schweinhart, A sharp deconfinement transition for Potts lattice gauge theory in codimension two, Commun. Math. Phys. 406, 164 (2025).
	
    \bibitem{fv2025} M. P. Forsstr\"om, F. Viklund, A current expansion for Ising lattice gauge theory, preprint available at \url{https://arxiv.org/abs/2502.19942} (2025).
	
    \bibitem{fv2025b} M. P. Forsstr\"om, F. Viklund, Free energy and quark potential in Ising lattice gauge theory via cluster expansion, Int. Math. Res. Not, 2025(12). 

    \bibitem{f2024a} M. P. Forsstr\"om, Pure perimeter laws for Wilson~lines observables, preprint available at \url{https://arxiv.org/abs/2409.20085} (2024).
    
    \bibitem{f2024b} M. P. Forsstr\"om, The phase transition of the Marcu--Fredenhagen ratio in the abelian lattice Higgs model, Electron. J. Probab. 29: 1--36 (2024). 

    \bibitem{f2025} M. P. Forsstr\"om, Ornstein-Zernike decay of Wilson~line observables in the free phase of the \( \mathbb{Z}_2 \) lattice Higgs model, preprint available at \url{https://arxiv.org/abs/2504.10909} (2025).
 
    \bibitem{f2022b} M. P. Forsstr\"om, Wilson~lines in the Abelian lattice Higgs model, Comm. Math. Phys., Volume 405, article number 275, (2024).
     
 	\bibitem{flv} M. P. Forsstr\"om, J. Lenells, F. Viklund,  Wilson~lines in the lattice Higgs model at strong coupling, Ann. Appl. Probab. 35(1): 590--634 (2025).  
 
	\bibitem{flv2022} M. P. Forsstr\"om, J. Lenells, F. Viklund, Wilson~loops in finite abelian lattice gauge theories, Ann. Inst. Henri Poincar\'e, Probab. Stat. Vol. 58, Issue 4, (2022), 2129--2164.

	\bibitem{fs1979} E. Fradkin, S. Shenker, Phase diagrams of lattice gauge theories with Higgs fields, Phys. Rev. D Vol. 19, No. 12 (1979)


	\bibitem{fm1983} K. Fredenhagen, M. Marcu, Charged states in $\mathbb{Z}_2$ gauge theories Commun. Math. Phys. 92, 81--119 (1983).


	\bibitem{fm1988}  K. Fredenhagen, M. Marcu, Dual interpretation of order parameters for lattice gauge theories with matter fields, Nucl.{} Phys.{} B (Proc. Suppl.) 4 (1988) 352--357.

	\bibitem{fv2007} S. Friedli, Y. Velenik, Statistical mechanics of lattice systems: A Concrete Mathematical Introduction, Cambridge: Cambridge University Press, (2017), ISBN: 978-1-107-18482-4, DOI: 10.1017/9781316882603

	\bibitem{fs1982} J. Fr\"ohlich, T. Spencer, Massless phases and symmetry restoration in abelian gauge theories and spin systems, Comm. Math. Phys. 83(3): 411-454 (1982).

	\bibitem{gs2023} C. Garban, A. Supelveda, Improved spin-wave estimate for Wilson~loops in U(1) lattice gauge theory, Int.{} Math.{} Res.{} Not. (2023) 21, 18142--18198.

	\bibitem{g2011} K. Gregor, D. Huse, R. Moessner, S. L. Sondhi, Diagnosing deconfinement and topological order, New J.{} Phys.{} 13 025009 (2011).
 
    
    
	
	\bibitem{g1980} A. H. Guth, Existence Proof of a Nonconfining Phase in Four-Dimensional U(1) Lattice Gauge Theory,  Phys.{} Rev.{} D 21 (1980) 2291  
 
	\bibitem{gm1982} M. G\"opfert, G. Mack, Proof of confinement of static quarks in 3-dimensional U(1) lattice gauge theory for all values of the coupling constant, Comm.{} Math.{} Phys.{} 82(4): 545-606 (1981-1982).
	

	\bibitem{k1985} K. Kondo,  Order parameter for charge confinement and phase structures in the lattice U(1) gauge-Higgs model, Prog.{} Theor.{} Phys.{} 74, 152--169 (1985) 

	\bibitem{k1987} K. Kondo, Rigorous analyses on order parameters and charged states in the lattice abelian-Higgs model, Prog. Theor. Phys. Supplement No. 92, (1987), 72--107.

	
	\bibitem{os1978} K. Osterwalder, E. Seiler, Gauge Field Theories on a Lattice, Ann.{} Phys.{} 110, 440--471 (1978).
	

\bibitem{ssn2021} A. M. Somoza, P.  Serna, A. Nahum, Self-dual criticality in three-dimensional \( \mathbb{Z}_2 \) gauge theory with matter, Phys. Rev. X 11 (4), (2021).

\bibitem{spkl2026} C. Stahl, B. Placke, V. Khemani, Y. Li,
 Slow mixing and emergent one-form symmetries in three-dimensional \( \mathbb{Z}_2 \) gauge theory, preprint available as 	arXiv:2601.06010 (2026).
 

	\bibitem{w1971} F. J. Wegner, Duality in generalized Ising models and phase transitions without local order parameters, J.\ Math.\ Phys.\ 12(10) (1971), 2259--2272.
    
    \bibitem{w1974} K. Wilson, {Confinement of quarks}, Phys.\ Rev.\ D 10, (1974), 2445--2459. 
 



	


    

	
	
    




     
 \end{thebibliography}
\end{document}